\documentclass[12pt,a4paper,reqno]{amsart} %Fr o m detta
\usepackage{amssymb}
\usepackage{latexsym}
\usepackage{amsmath}
\usepackage{mathrsfs}
\usepackage[T1]{fontenc} % \usepackage[X2,T1]{fontenc}
\usepackage{cite}
\usepackage{calc}                   %Detta beh?vs f?r att
\newcommand{\scal}[2]{\langle #1,#2\rangle}
\newcommand{\rr}[1]{\mathbf R^{#1}}

\newcommand{\mabfq}{{\boldsymbol q}}

\newcommand{\zz}[1]{\mathbf Z^{#1}}
\newcommand{\nm}[2]{\Vert #1\Vert _{#2}}
\newcommand{\Nm}[2]{\left \Vert #1 \right \Vert _{#2}}

\newcommand{\sets}[2]{\{ \, #1\, ;\, #2\, \} }

\newcommand{\ep}{\varepsilon}
\newcommand{\fy}{\varphi}
\newcommand{\cdo}{\, \cdot \, }

\newcommand{\esssup}{\operatorname{ess\, sup}}

\newcommand{\eabs}[1]{\langle #1\rangle}     %%%%%   for <x>

\newcommand{\vrum}{\vspace{0.1cm}}

\newcommand{\Fo}[1]{\operatorname {FS}_{#1}}

\newcommand{\indlim}{\operatorname{ind \, lim\, }}
\renewcommand{\projlim}{\operatorname{proj \, lim\, }}

\newcommand{\rd}{\mathbf{R} ^{d}}

\newcommand{\nn}[1]{{\mathbf N}^{#1}}

\newcommand{\maclE}{\mathcal E}
\newcommand{\maclG}{\mathcal G}

\newcommand{\maclS}{\mathcal S}
\newcommand{\mascB}{\mathscr B}

\newcommand{\mascE}{\mathscr E}
\newcommand{\mascF}{\mathscr F}

\newcommand{\mascP}{\mathscr P}
\newcommand{\mascS}{\mathscr S}

\numberwithin{equation}{section}          %Detta g?r att man f?r
\newtheorem{thm}{Theorem}
\numberwithin{thm}{section}

\newcommand{\rubrik}{}
\newtheorem{prop}[thm]{Proposition}
\newtheorem{cor}[thm]{Corollary}
\newtheorem{lemma}[thm]{Lemma}

\theoremstyle{definition}

\newtheorem{defn}[thm]{Definition}
\newtheorem{example}[thm]{Example}

\theoremstyle{remark}
\newtheorem{rem}[thm]{Remark}

\author{Joachim Toft}

\address{Department of Mathematics,
Linn{\ae}us University, V{\"a}xj{\"o}, Sweden}

\email{joachim.toft@lnu.se}

\author{Elmira Nabizadeh}

\address{Department of Mathematics,
Linn{\ae}us University, V{\"a}xj{\"o}, Sweden}

\email{elmira.nabizadeh.extern@lnu.se}

\title{Periodic distributions and periodic elements in modulation spaces}


\begin{document}

\begin{abstract}
We characterize periodic elements in Gevrey classes,
Gelfand-Shilov distribution spaces and modulation spaces, in
terms of estimates of involved Fourier coefficients, and by estimates of
their short-time Fourier transforms.
If $q\in [1,\infty )$, $\omega$ is a suitable weight
and $(\maclE _0^E)'$ is the set of all $E$-periodic elements, then we prove that
the dual of $M^{\infty ,q}_{(\omega )}\cap (\maclE _0^E)'$ equals
$M^{\infty ,q'}_{(1/\omega )}\cap (\maclE _0^E)'$ by
suitable extensions of Bessel's identity.
\end{abstract}

\maketitle

\par

%%%%%%%%%%%%%%%%%%%%%%%
\section{Introduction}\label{sec0}
%%%%%%%%%%%%%%%%%%%%%%%

\par

A fundamental issue in analysis concerns periodicity.
For example, several problems in the theory of
partial differential equations and in signal processing involve
periodic functions and distributions. In such situations it is in general
possible to discretize the problems by means of Fourier series expansions
of these functions and distributions.

\par

We recall that if $f$ is a smooth $1$-periodic function on $\rr d$, then $f$ is equal to
its Fourier series
\begin{equation}\label{Eq:IntrFourExp}
\sum _{\alpha \in \zz d}c(\alpha )e^{2\pi i \scal \cdo \alpha},
\end{equation}
where the Fourier coefficients
$c(\alpha )$ can be evaluated by the formula
$$
c(f,\alpha )= c(\alpha ) = \int _{[0,1]^d}f(x)e^{-2\pi i\scal x\alpha}\, dx . 
$$
(Our investigations later on involve functions and distributions with more general periodics.
See also \cite{Ho1}, and Sections \ref{sec1} and \ref{sec2} for notations.)
By the smoothness of $f$ it follows that for every $N\ge 0$, there is a constant
$C_N\ge 0$ such that
\begin{equation}\label{Eq:PerSmoothCoeffCond}
|c(\alpha )| \le C_N\eabs \alpha ^{-N},
\end{equation}
and it follows from Weierstrass theorem that the series \eqref{Eq:IntrFourExp}
is uniformly convergent (cf. e.{\,}g. \cite[Section 7.2]{Ho1}).

\par

Assume instead that $f$ is a $1$-periodic distribution on $\rr d$, and let $\phi $
be compactly supported and smooth on $\rr d$ such that
\begin{equation}\label{Eq:PartUnity}
\sum _{k\in \zz d} \phi (\cdo -k) =1.
\end{equation}
Then $f$ is a tempered distribution and is still equal to its Fourier series
\eqref{Eq:IntrFourExp} in distribution sense. The Fourier coefficients
for $f$ are uniquely defined and can be computed by
\begin{equation}\label{Eq:FourCoeffHorm}
c(f,\alpha )= c(\alpha ) = \scal f{\phi e^{-i\scal \cdo \alpha }},
\end{equation}
and satisfy
\begin{equation}\label{Eq:PerDistCoeffCond}
|c(\alpha )| \le C\eabs \alpha ^{N},
\end{equation}
for some constants $C$ and $N$ which only depend on $f$. (Cf. e.{\,}g.
\cite[Section 7.2]{Ho1}. See also \cite{Zi} for an early approach to formal
 Fourier series expansions.)

\par

The conditions \eqref{Eq:PerSmoothCoeffCond} and
\eqref{Eq:PerDistCoeffCond} are not only necessary but also sufficient for
a formal Fourier series expansion \eqref{Eq:IntrFourExp} being smooth
respectively a tempered distribution. Hence, by a unique extension of
Parseval's identity
\begin{equation}\label{Eq:PerForm}
(f,\phi )_1 \equiv \sum _{\alpha \in \zz d} c(f,\alpha )\overline{c(\phi ,\alpha )}
= \int _{[0,1]^d} f(x)\overline{\phi (x)}\, dx
\end{equation}
on smooth $1$-periodic functions on $\rr d$, it follows that the dual of
the set of smooth $1$-periodic functions on $\rr d$ is the set of all
$1$-periodic tempered distributions on $\rr d$.

\par

Some investigations of periodicity in the framework of ultra-differentiability have also
been performed. More precisely, let $s>0$ and $\maclE _s^1(\rr d)$ be
the set of all $1$-periodic
functions in the Gevrey class $\maclE _s(\rr d)$ of Roumieu type. (Our investigations
later on also involve Gevrey classes of Beurling type.)
It is proved in \cite{Pil2} by Pilipovi{\'c} that a smooth $1$-periodic
function $f$ on $\rr d$ belongs
$\maclE _s^1(\rr d)$, if and only if its Fourier coefficients satisfy
$$
|c(\alpha )| \le Ce^{-r|\alpha |^{\frac 1s}},
$$
for some constants $C>0$ and $r>0$ which are independent of $\alpha$.

\par

Due to straight-forward extensions of Parseval's identity it follows that the dual
$(\maclE _s^1)'(\rr d)$
of $\maclE _s^1(\rr d)$ can be identify with all expansions \eqref{Eq:IntrFourExp} such that
for every constants $r>0$ there is a constant $C>0$ such that
$$
|c(\alpha )| \le Ce^{r|\alpha |^{\frac 1s}}.
$$
In the case $s\ge 1$, it seems to be shown by Gorba{\v c}uk and Gorba{\v c}uk in \cite{Go,GoGo},
and commented in \cite{Pil2} that the set of such formal Fourier series expansions coincide
with the set of $1$-periodic Gelfand-Shilov distributions $(\maclS _t^s)'(\rr d)$ when $t>1$.

\par

The previous properties have been extended and explained in different ways,
see e.{\,}.g. \cite[Theorem 2.3]{DaRu1} and \cite{DaRu2,FiRu,GaRu,RuTu1,RuTu2,RuTu3} by
Dasgupta, Fischer, Garetto, Ruzhansky and Turunen. For example,
in \cite{DaRu2}, it is shown that characterizations of the previous types also hold on
more general manifolds, e.{\,}g. compact ones. Here we remark that such (global)
characterizations of Gevrey spaces and ultradistributions in terms of Fourier coefficients are
used to prove the well-posedness and estimates for solutions to wave equations for
H{\"o}rmander's sums of squares in \cite{GaRu}.

\par

The aim of the paper is obtain analogous and other characterizations for periodic
functions in Gevrey classes, and for
periodic ultra-distributions. Especially we characterize periodic Gelfand-Shilov
distributions and periodic elements in modulation spaces, in terms of estimates
of their Fourier coefficients. At the same time we deduce an integral formula
for evaluating the Fourier coefficients, and which involve the short-time Fourier
transforms of the involved periodic distributions. Finally we show that the duals
of the periodic functions in Gevrey classes can be identified with
suitable classes of periodic Gelfand-Shilov distributions, and characterize
elements in these classes by suitable estimates on the short-time Fourier
transforms of the involved functions. In contrast to earlier contributions, our
characterizations hold when the Gevrey parameters belong to the interval
$(0,\infty )$ instead of the sub interval $[1,\infty )$.

%\par
%
%
%
%
%
%
%
%In Section \ref{sec2} we
%present a short and straight-forward proof, based on characterizations of
%Gelfand-Shilov spaces by means of suitable estimates on the elements,
%their Fourier transforms and their short-time Fourier
%transforms (cf. \cite{ChuChuKim,Eij,GZ,Toft18}).
%
%\par
%
%The dual situation is handled by suitable extensions of the Bessel's identity
%\eqref{Eq:PerForm} on $\maclE _s^1(\rr d)$. In fact, in Section
%\ref{sec2} we clarify that by such extensions, the dual $(\maclE _s^1)'(\rr d)$
%of $\maclE _s^1(\rr d)$ can be
%identified with the set of formal
%$1$-periodic Fourier series expansions
%%%
%\begin{equation}\label{Eq:IntrFourExp}
%\sum _{\alpha \in \zz d}c(\alpha )e^{2\pi i \scal \cdo \alpha}
%\end{equation}
%%%
%such that for every $r>0$
%there is a constant $C=C_r>0$ such that
%$$
%|c(\alpha )| \le Ce^{r|\alpha |^{\frac 1s}}.
%$$
%%
%%In Section \ref{sec2} we show that the Bessel's identity
%%$$
%%(f,\phi ) \equiv \sum _{\alpha \in \zz d} c(f,\alpha )\overline{c(\phi ,\alpha )}
%%= \int _{[0,1]^d} f(x)\overline{\phi (x)}\, dx
%%$$
%%extends to a duality between periodic elements in $\maclE _s(\rr d)$ and formal
%%Fourier series expansions whose Fourier coefficients satisfies that for every $r>0$
%%there is a constant $C=C_r>0$ such that
%%$$
%%|c(\alpha )| \le Ce^{r|\alpha |^{\frac 1s}}.
%%$$
%Note that formal trigonometric series expansions were considered in \cite{Zi},
%and that such characterizations are obtained in \cite{Go,Pil2} in restricted case when $s>1$.
%Furthermore, for $s\ge 1$ a more general result can be found in \cite{DaRu1}
%(cf. \cite[Theorem 2.5]{DaRu}).

\par

In Section \ref{sec2} we deduce other characterizations of
$\maclE _s^1(\rr d)$ and $(\maclE _s^1)'(\rr d)$. For example let
$\phi$ be a non-zero element in the Gelfand-Shilov space $\maclS _t^s(\rr d)$.
Then we show that
$f\in \maclE _s^1(\rr d)$, if and only if
$f$ is $1$-periodic ultra-distribution and that its short-time
Fourier transform $V_\phi f$ satisfies
$$
|V_\phi f(x,\xi )|\le Ce^{-r|\xi |^{\frac 1s}}
$$
for some constants $C>0$ and $r>0$. In the same way we show that
$f\in (\maclE _s^1)'(\rr d)$, if and only if
$f$ is $1$-periodic ultra-distribution and for every $r>0$ there is a constant
$C>0$ such that
$$
|V_\phi f(x,\xi )|\le Ce^{r|\xi |^{\frac 1s}}.
$$
At the same time we show (for any $s>0$) that $(\maclE _s^1)'(\rr d)$
may in canonical ways be identified with the set of
periodic elements in the Gelfand-Shilov distribution space $(\maclS _t^s)'(\rr d)$,
provided $t>0$ satisfies $s+t\ge 1$.

\par

An ingredient in the proofs of these properties is the formula
\begin{equation}\label{Eq:IntrSTFTDescrForm}
(f,\psi )_1 = \nm \phi {L^2}^{-2}
\int _{[0,1]^d}
\left (
\int _{\rr d} (V_\phi f)(x,\xi )\overline{(V_\phi \psi)(x,\xi )}\, d\xi
\right )
\, dx ,
\end{equation}
proved in Section \ref{sec2} when evaluating the form in \eqref{Eq:PerForm}. By letting
$\psi = e^{2\pi i\scal \cdo \alpha}$, it follows by straight-forward computations
that \eqref{Eq:IntrSTFTDescrForm} takes the form
$$
c(f,\alpha)
= \nm \phi{L^2({\rd})}^{-2} \int _{[0,1]^d} \left(\int _{\rd} (V_\phi f)(x,\xi)
\widehat{\phi}(\alpha - \xi) e^{-2\pi i\scal x {\alpha - \xi}} \,d\xi \right ) \, dx.
$$
Here the integrand belongs to $L^1([0,1]^d\times \rr d)$ due the deduced characterizations
of $(\maclE _s^1)'(\rr d)$.

\par

It seems to be difficult to find the previous formulae in the literature.
When using \eqref{Eq:FourCoeffHorm} to compute the Fourier coefficients,
it is essential that $\phi$ satisfies \eqref{Eq:PartUnity}. For these
reasons it is difficult to carry over \eqref{Eq:FourCoeffHorm}
to the Gevrey or Gelfrand-Shilov situation when $s$
above is less than $1$, since it is difficult to find $\phi \in \maclS _t^s(\rr d)$ which satisfies
\eqref{Eq:PartUnity}.

\par

In Section \ref{sec3} we characterize periodic distributions in modulation spaces.
In particular we deduce that if $q\in (0,\infty ]$ and
$\omega (x,\xi )=\omega _0(\xi )$ is a suitable weight on $\rr d$, then
the $1$-periodic elements in the modulation spaces
$M^{\infty ,q}_{(\omega )}$ and $W^{\infty ,q}_{(\omega )}$ agree and are equal to
the set of formal Fourier series expansions in \eqref{Eq:IntrFourExp}
such that
$$
\{ c(\alpha )\omega _0(\alpha )\} _{\alpha \in \zz d} \in \ell ^q.
$$
In particular we extend Proposition 2.6 in \cite{Re} and Proposition 5.1 in
\cite{RuSuToTo} to involve more general
weights and permit $q$ to be in the broader interval $(0,\infty ]$ instead of
$[1,\infty ]$.

\par

In the last part of Section \ref{sec3} we apply these results to deduce that if
$q\in [1,\infty )$ and $\frac 1q+\frac 1q' =1$, then the dual of $M^{\infty ,q}_{(\omega )}(\rr d)
\cap (\maclE _0^1)'(\rr d)$ is equal to $M^{\infty ,q'}_{(1/\omega )}(\rr d)
\cap (\maclE _0^1)'(\rr d)$ through suitable extensions of the form $(\cdo ,\cdo )_1$
on $\maclE _0^1(\rr d)\times \maclE _0^1(\rr d)$.

%We note that in contrast to
%
%for evaluating
%the form 
%which is
%not easy to find in the literature
%In Section \ref{sec2}
%
%
%
% Proposition 2.1.2 in \cite{FiRu}. Periodic functions/distributions in
% modulation spaces in \cite{ReReSi}.
%

\par

%%%%%%%%%%%%%%%%%%%%%%%
\section{Preliminaries}\label{sec1}
%%%%%%%%%%%%%%%%%%%%%%%

\par

In this section we recall some basic facts. We start by discussing
Gelfand-Shilov spaces and their properties. Thereafter we recall
some properties of modulation spaces and discuss different aspects
of periodic distributions

\par

\subsection{Gelfand-Shilov spaces and Gevrey classes}\label{subsec1.1}
%We start by recalling some facts on Gelfand-Shilov spaces.
Let $0<s,t\in \mathbf R$ be fixed. Then the Gelfand-Shilov
space $\mathcal S_{t}^s(\rr d)$
($\Sigma _{t}^s(\rr d)$) of Roumieu type (Beurling type) with parameters $s$
and $t$ consists of all $f\in C^\infty (\rr d)$ such that
\begin{equation}\label{gfseminorm}
\nm f{\mathcal S_{t,h}^s}\equiv \sup \frac {|x^\beta \partial ^\alpha
f(x)|}{h^{|\alpha  + \beta |}\alpha !^s\, \beta !^t}
\end{equation}
is finite for some $h>0$ (for every $h>0$). Here the supremum should be taken
over all $\alpha ,\beta \in \mathbf N^d$ and $x\in \rr d$. We equip
$\mathcal S_{t}^s(\rr d)$ ($\Sigma _{t}^s(\rr d)$) by the canonical inductive limit
topology (projective limit topology) with respect to $h>0$, induced by
the semi-norms in \eqref{gfseminorm}.

\par

For any $s,t,s_0,t_0>0$ such that $s>s_0$, $t>t_0$ and $s+t\ge 1$ we have
\begin{equation}\label{GSembeddings}
\begin{alignedat}{3}
\maclS _{t_0}^{s_0}(\rr d)
&\hookrightarrow &
\Sigma _{t}^{s}(\rr d)
&\hookrightarrow &
\maclS _t^s(\rr d)
&\hookrightarrow 
\mascS (\rr d),
\\[1ex]
\mascS '(\rr d)
&\hookrightarrow &
(\maclS _t^s)' (\rr d)
&\hookrightarrow &
(\Sigma _{t}^{s})'(\rr d)
&\hookrightarrow
(\maclS _{t_0}^{s_0})'(\rr d),
\end{alignedat}
\end{equation}
with dense embeddings.
Here and in what follows we use the notation $A\hookrightarrow B$ 
when the topological spaces $A$ and $B$ satisfy $A\subseteq B$ with
continuous embeddings.
The space $\Sigma _t^s(\rr d)$ is a Fr{\'e}chet space
with seminorms $\nm \cdo{\mathcal S_{t,h}^s}$, $h>0$. Moreover,
$\Sigma _t^s(\rr d)\neq \{ 0\}$, if and only if $s+t\ge 1$ and
$(s,t)\neq (\frac 12,\frac 12)$, and $\maclS _t^s(\rr d)\neq \{ 0\}$, if and only
if $s+t\ge 1$.

\medspace

The \emph{Gelfand-Shilov distribution spaces} $(\mathcal S_t^{s})'(\rr d)$
and $(\Sigma _t^s)'(\rr d)$ are the dual spaces of $\mathcal S_t^{s}(\rr d)$
and $\Sigma _t^s(\rr d)$, respectively.  As for the Gelfand-Shilov spaces there 
is a canonical projective limit topology (inductive limit topology) for $(\maclS _t^{s})'(\rr d)$ 
($(\Sigma _t^s)'(\rr d)$).(Cf. \cite{GS, Pil1, Pil3}.)
For conveniency we set
$$
\maclS _s=\maclS _s^s,\quad \maclS _s'=(\maclS _s^s)',\quad
\Sigma _s=\Sigma _s^s
\quad \text{and}\quad
\Sigma _s'=(\Sigma _s^s)'.
$$

\par

From now on we let $\mathscr F$ be the Fourier transform which
takes the form
$$
(\mathscr Ff)(\xi )= \widehat f(\xi ) \equiv (2\pi )^{-\frac d2}\int _{\rr
{d}} f(x)e^{-i\scal  x\xi }\, dx
$$
when $f\in L^1(\rr d)$. Here $\scal \cdo \cdo$ denotes the usual
scalar product on $\rr d$. The map $\mathscr F$ extends 
uniquely to homeomorphisms on $\mathscr S'(\rr d)$,
from $(\mathcal S_t^s)'(\rr d)$ to $(\mathcal S_s^t)'(\rr d)$ and
from $(\Sigma _t^s)'(\rr d)$ to $(\Sigma _s^t)'(\rr d)$. Furthermore,
$\mascF$ restricts to
homeomorphisms on $\mathscr S(\rr d)$, from
$\mathcal S_t^s(\rr d)$ to $\mathcal S_s^t(\rr d)$ and
from $\Sigma _t^s(\rr d)$ to $\Sigma _s^t(\rr d)$,
and to a unitary operator on $L^2(\rr d)$. 

\par

Gelfand-Shilov spaces can in convenient
ways be characterized in terms of estimates the functions and their Fourier
transforms. More precisely, in \cite{ChuChuKim, Eij} it is proved that
if $f\in \mascS '(\rr d)$ and $s,t>0$, then $f\in \maclS _t^s(\rr d)$
($f\in \Sigma _t^s(\rr d)$), if and only if
\begin{equation}\label{Eq:GSFtransfChar}
|f(x)|\lesssim e^{-r|x|^{\frac 1t}}
\quad \text{and}\quad
|\widehat f(\xi )|\lesssim e^{-r|\xi |^{\frac 1s}},
\end{equation}
for some $r>0$ (for every $r>0$). Here and in what follows, $A\lesssim B$
means that $A\le cB$ for a suitable
constant $c>0$. We also set $A\asymp B$ when $A\lesssim B$
and $B\lesssim A$.

\par

Gelfand-Shilov spaces and their distribution spaces can also
be characterized by estimates of short-time Fourier
transforms, (see e.{\,}g. \cite{GZ,Toft18}).
More precisely, let $\phi \in \maclS _s (\rr d)$ be fixed. Then the \emph{short-time
Fourier transform} $V_\phi f$ of $f\in \maclS _s '
(\rr d)$ with respect to the \emph{window function} $\phi$ is
the Gelfand-Shilov distribution on $\rr {2d}$, defined by
$$
V_\phi f(x,\xi )  =
\mascF (f \, \overline {\phi (\cdo -x)})(\xi ).
$$
If $f ,\phi \in \maclS _s (\rr d)$, then it follows that
$$
V_\phi f(x,\xi ) = (2\pi )^{-\frac d2}\int f(y)\overline {\phi
(y-x)}e^{-i\scal y\xi}\, dy .
$$

\par

\begin{rem}
Let $s_1,s_2,t_1,t_2>0$. Then the 
Gelfand-Shilov space $\maclS_{t_1,t_2}^{s_1,s_2}(\rr {2d})$
($\Sigma_{t_1,t_2}^{s_1,s_2}(\rr {2d})$)
is the set of all $F\in C^{\infty}(\rr {2d})$ such that 
\begin{equation*}
\nm {x_1^{\alpha_1} x_2^{\alpha_2} \partial_{x_1}^{\beta_1}
\partial_{x_2}^{\beta_2} F}{L^{\infty}}
\lesssim h^{|\alpha_1+\alpha_2+\beta_1+\beta_2|}
\alpha_1 !^{t_1}\alpha_2 !^{t_2}\beta_1 !^{s_1}\beta_2 !^{s_2}
\end{equation*}
for some $h>0$ (for every $h>0$). 

\par

We also let $(\maclS_{t_1,t_2}^{s_1,s_2})'(\rr {2d})$
($(\Sigma_{t_1,t_2}^{s_1,s_2})'(\rr {2d})$) be corresponding duals
(distribution spaces).

\par

By \cite[Theorem 2.3]{Toft10} it follows that the definition of the map
$(f,\phi)\mapsto V_{\phi} f$ from $\mascS (\rd) \times \mascS (\rd)$ 
to $\mascS(\rr {2d})$ is uniquely extendable to a continuous map from 
$(\maclS_{t}^{s})'(\rd)\times (\maclS_{t}^{s})'(\rd)$
to $(\maclS_{t,s}^{s,t})'(\rr {2d})$, and restricts to a continuous map
from $\maclS_{t}^{s}(\rd)\times \maclS_{t}^{s}(\rd)$
to $\maclS_{t,s}^{s,t}(\rr {2d})$.

\par

The same conclusion holds with $\Sigma_t^s$ and $\Sigma_{t,s}^{s,t}$ 
in place of $\maclS_{t}^{s}$ and $\maclS_{t,s}^{s,t}$, respectively, at each place.
\end{rem}

\par

The following properties characterize Gelfand-Shilov spaces and their distribution
spaces in terms of estimates of short-time fourier transform.

\par

\begin{prop}\label{stftGelfand2}
Let $s,t>0$ be such that $s+t\ge 1$. Also let
$\phi \in \mathcal S_{t}^{s}(\rr d)\setminus 0$ ($\phi \in \Sigma_{t}^{s}(\rr d)\setminus 0$) and
$f$ be a Gelfand-Shilov distribution on $\rd$.
Then $f\in  \mathcal S_{t}^s(\rr d)$ ($f\in  \Sigma_{t}^s(\rr d)$), if and only if
\begin{equation}\label{stftexpest2}
|V_\phi f(x,\xi )| \lesssim  e^{-r (|x|^{\frac 1t}+|\xi |^{\frac 1s})},
\end{equation}
for some $r > 0$ (for every $r>0$).
\end{prop}

\par

\begin{prop}\label{stftGelfand2dist}
Let $s,t>0$ be such that $s+t\ge 1$. Also let
$\phi \in \mathcal S_{t}^{s}(\rr d)\setminus 0$ ($\phi \in \Sigma_{t}^{s}(\rr d)\setminus 0$) and
$f$ be a Gelfand-Shilov distribution on $\rd$.
Then $f\in  (\mathcal S_{t}^s)'(\rr d)$($f\in (\Sigma^s_{t})'(\rd)$), if and only if
\begin{equation}\label{stftexpest2Dist}
|V_\phi f(x,\xi )| \lesssim  e^{r(|x|^{\frac 1t}+|\xi |^{\frac 1s})},
\end{equation}
for every $r > 0$ (for some $r > 0$).
\end{prop}

\par

We note that if $s=t=\frac 12$ in Propositions \ref{stftGelfand2} and \ref{stftGelfand2dist},
then it is not possible to find any $\phi \in \Sigma _t^s(\rr d)\setminus 0$. Hence, these
results give no information in the Beurling case for such choices of $s$ and $t$.

\par

A proof of Proposition \ref{stftGelfand2} can be found in
e.{\,}g. \cite{GZ} (cf. \cite[Theorem 2.7]{GZ}) and a proof of
Proposition \ref{stftGelfand2dist} in the general situation
can be found in \cite{Toft18}. See also \cite{CPRT10} for related results.

\par

In Section \ref{sec2} we deduce analogous characterizations for
periodic functions and distributions.

\par

\begin{rem}\label{SchwFunctionSTFT}
The short-time Fourier transform can also be used to identify
elements in $\mascS (\rr d)$ and in $\mascS '(\rr d)$. In
fact, if $\phi \in \mascS (\rr d)\setminus 0$ and $f$ is a Gelfand-Shilov
distribution on $\rr d$, then the following is true:
\begin{enumerate}
\item $f\in \mascS (\rr d)$, if and only if for every $N\ge 0$,
it holds
$$
|V_\phi f(x,\xi )| \lesssim \eabs {(x,\xi )}^{-N} \text ;
$$

\vrum

\item $f\in \mascS '(\rr d)$, if and only if for some $N\ge 0$,
it holds
$$
|V_\phi f(x,\xi )| \lesssim \eabs {(x,\xi )}^{N} .
$$
\end{enumerate}
(Cf. \cite[Chapter 12]{Gc2}.)
\end{rem}

\par

Next we consider Gevrey classes on $\rd$. Let $s\ge 0$.
For any compact set $K\subseteq \rd$, $h>0$ and $f\in C^{\infty}(\rd)$ let
\begin{equation}\label{e2}
\nm {f}{K,h,s} \equiv \underset{\alpha\in \nn d}\sup
\frac{\nm {\partial^{\alpha}f}{L^{\infty}(K)}}{h^{\vert \alpha\vert}\alpha ! ^s}.
\end{equation} 
The Gevrey class $\maclE _s(K)$ ($\maclE _{0,s}(K)$) of order $s$ and
of Roumieu type (of Beurling type) is the set of all
$f\in C^{\infty}(K)$ such that \eqref{e2} is finite for some (for every)
$h>0$. We equipp $\maclE _s(K)$ ($\maclE _{0,s}(K)$) by
the inductive (projective) limit topology supplied by the seminorms
in \eqref{e2}. Finally if $\lbrace K_j\rbrace_{j\geq 1}$ is an exhausted
sets of compact subsets of $\rd$, then let
\begin{alignat*}{3}
\maclE _s(\rd) &= \underset{j}\projlim \maclE _s(K_j)
& \quad &\text{and} &\quad
\maclE _{0,s}(\rd) &= \underset{j}\projlim \maclE _s(K_j).
\intertext{In particular,}
\maclE _s(\rd) &=\underset{j\geq 1}\bigcap \maclE _s(K_j)
& \quad &\text{and} &\quad
\maclE _{0,s}(\rd) &= \underset{j\geq 1}\bigcap \maclE _{0,s}(K_j).
\end{alignat*}
It is clear that $\maclE _{0}(\rr d)$ contains
all trigonometric polynomials, which is not the case for $\maclE _{0,0}(\rr d)$.
% consists of all constant functions 
%on $\rr d$, while
%
%\maclE _{0,0}(\rr d) contains all polynomials.
%

\par

\subsection{Modulation spaces}\label{subsec1.2}

\par

We consider a general class of modulation spaces (cf. \cite{Fei5}),
and begin with discussing general properties for the involved weight
functions. A \emph{weight} on $\rr d$ is a positive function $\omega
\in  L^\infty _{loc}(\rr d)$ such that $1/\omega \in  L^\infty _{loc}(\rr d)$.
%, and for each compact set $K\subseteq
%\rr d$, there is a constant $c>0$ such that
%$$
%\omega (x)\ge c\qquad \text{when}\qquad x\in K.
%$$
A usual condition on $\omega$ is that it should be \emph{moderate},
or \emph{$v$-moderate} for some positive function $v \in
 L^\infty _{loc}(\rr d)$. This means that
\begin{equation}\label{moderate}
\omega (x+y) \lesssim \omega (x)v(y),\qquad x,y\in \rr d.
\end{equation}
We note that \eqref{moderate} implies that $\omega$ fulfills
the estimates
\begin{equation}\label{moderateconseq}
v(-x)^{-1}\lesssim \omega (x)\lesssim v(x),\quad x\in \rr d.
\end{equation}
We let $\mascP _E(\rr d)$ be the set of all moderate weights on $\rr d$.
%Furthermore, if $v$ in \eqref{moderate} can be chosen as a polynomial,
%then $\omega$ is called \emph{polynomially moderate}, or a weight of
%\emph{polynomial type}. We let
%$\mascP (\rr d)$ be the set of all weights of polynomial type.

\par

It can be proved that if $\omega \in \mascP _E(\rr d)$, then
$\omega$ is $v$-moderate for some $v(x) = e^{r|x|}$, provided the
positive constant $r$ is large enough (cf. \cite{Gc2.5}). In particular,
\eqref{moderateconseq} shows that for any $\omega \in \mascP
_E(\rr d)$, there is a constant $r>0$ such that
$$
e^{-r|x|}\lesssim \omega (x)\lesssim e^{r|x|},\quad x\in \rr d.
$$

\par

We say that $v$ is
\emph{submultiplicative} if $v$ is even and \eqref{moderate}
holds with $\omega =v$. In the sequel, $v$ and $v_j$ for
$j\ge 0$, always stand for submultiplicative weights if
nothing else is stated.

\par

\begin{defn}\label{Def:InvSpaces}
Let $r\in (0,1]$, $v\in \mascP _E(\rr {d})$ 
and let $\mascB \subseteq L^r_{loc}(\rr {d})$
be a quasi-Banach space. Then $\mascB$ is called $v$-invariant on $\rr d$ if
the following is true:
\begin{enumerate}
\item $x\mapsto f(x+y)$ belongs to $\mascB$ for every $f\in \mascB$
and $y\in \rr {d}$.

\vrum

\item There is a constant $C>0$ such that $\nm {f_1}{\mascB}\le C\nm {f_2}{\mascB}$
when $f_1,f_2\in \mascB$ are such that $|f_1|\le |f_2|$. Moreover,
$$
\nm {f(\cdo +y)}{\mascB}\lesssim \nm {f}{\mascB}v(y),\qquad
f\in \mascB ,\ y\in \rr {d}.
$$
\end{enumerate}
\end{defn}

\par

The quasi-Banach spaces in the previous definition is usually a
mixed quasi-normed Lebesgue space, given in Definition \ref{Def:MixedLebSpaces}
below.
Here $\operatorname S_d$ is the set of permutations on $\{ 1,\dots ,d\}$,
$$
\max \mabfq =\max (q_1,\dots ,q_d)
\quad \text{and}\quad
\min \mabfq =\min (q_1,\dots ,q_d),
$$
when $\mabfq =(q_1,\dots ,q_d)\in (0,\infty ]^d$.

\par

Here we also let $E$ be a non-degenerate parallelepiped in $\rd$. That is, there is a basis
$e_1,\dots,e_d$ of $\rr d$ such that
$$
E = \sets{x_1e_1+\cdots+x_de_d}{(x_1,\dots,x_d)\in\rd,\ 0\leq x_k\leq 1,\
k=1,\dots,d}.
$$
The corresponding lattice, dual parallelepiped and dual lattice are given by
\begin{align*}
\Lambda _E &=\sets{j_1e_1+\cdots +j_de_d}{(j_1,\dots,j_d)\in \zz d},
\\[1ex]
E' &=\sets {\xi _1e'_1+\cdots+\xi _de'_d}{(\xi _1,\dots ,\xi _d)\in\rd,
\ 0\leq \xi _k\leq 1,\ k=1,\dots ,d},
\intertext{and}
\Lambda'_E &= \Lambda_{E'}=\sets{\iota _1e'_1+\cdots +\iota _de'_d}
{(\iota _1,\dots ,\iota _d) \in \zz d},
\intertext{respectively, where $e'_1,\dots ,e'_d$ satisfies}
\scal {e_j} {e'_k} &= 2\pi \delta_{jk}
\quad \text{for every}\quad
j,k =1,\dots, d.
\end{align*}
Evidently, $e'_1,\dots ,e'_d$ is a basis of $\rr d$. It is called the \emph{dual basis } of
$e_1,\dots ,e_d$.
We observe that there is a matrix $T_E$ with $e_1,\dots ,e_d$ as the image of
the standard basis, and that the image of the standard basis of $T_{E'}= 2\pi(T^{-1}_E)^t$
is given by $e_1',\dots ,e_d'$.

\par

\begin{defn}\label{Def:MixedLebSpaces}
Let $E$ be a non-degenerate parallelepiped in $\rr d$, $E'$ be the dual parallelepiped
spanned by the ordered set $\mathcal O_0=( e_1',\dots ,e_d')$
in $\rr d$, $\mabfq =(q_1,\dots ,q_d)\in (0,\infty ]^{d}$, $r=\min (1,\mabfq )$ and
$\tau \in \operatorname S_{d}$. If $a\in \ell _0'(\Lambda _E')$
and $f\in L^r_{loc}(\rr d)$, then
$$
\nm a{\ell ^{\mabfq }_{\mathcal O_0,\tau}}\equiv
\nm {b_{d-1}}{\ell ^{q_{d}}(\mathbf Z)}
\quad \text{and}\quad
\nm f{L^{\mabfq }_{\mathcal O_0,\tau}}\equiv
\nm {g_{d-1}}{L^{q_{d}}(\mathbf R)}
$$
where $b_k(l_k)$ and $g_k(z_k)$, $l_k\in \zz {d-k}$ and $z_k\in \rr {d-k}$,
$k=0,\dots ,d-1$, are inductively defined as
\begin{align*}
b_0(j_1,\dots ,j_{d}) &\equiv |a(j_1e_{\tau (1)}+\cdots +j_{d}e_{\tau (d)})|,
\\[1ex]
g_0(x_1,\dots ,x_{d}) &\equiv |f(x_1e_{\tau (1)}+\cdots +x_{d}e_{\tau (d)})|,
\\[1ex]
b_k(l_k) &\equiv
\nm {b_{k-1}(\cdo ,l_k)}{\ell ^{q_k}(\mathbf Z)},
\intertext{and}
g_k(z_k) &\equiv
\nm {g_{k-1}(\cdo ,z_k)}{L^{q_k}(\mathbf R)},
\quad k=1,\dots ,d-1.
\end{align*}
The space $\ell ^{\mabfq }_{\mathcal O_0,\tau}(\Lambda _E')$
consists of all $a\in \ell _0'(\Lambda _E')$ such that
$\nm a{\ell ^{\mabfq}_{\mathcal O_0,\tau}}$ is finite,
and $L^{\mabfq }_{\mathcal O_0,\tau}(\rr d)$ consists
of all $f\in L^r_{loc}(\rr d)$ such that
$\nm f{L^{\mabfq}_{\mathcal O_0,\tau}}$ is finite.
\end{defn}

\par

\begin{defn}\label{Def:ModSpaces}
Let $\omega ,v\in \mascP _E(\rr {2d})$ be such that $\omega$ is $v$-moderate,
$\mascB$ be a $v$-invariant space on
$\rr {2d}$, and let $\phi \in \maclS _{1/2}(\rr d)\setminus 0$.
Then the \emph{modulation space} $M(\omega ,\mascB)$ consists of all $f\in \maclS _{1/2}'
(\rr d)$ such that
\begin{equation}\label{Eq:ModSpNorm}
\nm f{M(\omega ,\mascB)} \equiv \nm {V_\phi f \cdot \omega}{\mascB}
\end{equation}
is finite.
\end{defn}

\par

The theory of modulation spaces has developed in different ways since they
were introduced in \cite{F1} by Feichtinger. (Cf. e.{\,}g. \cite{Fei5,GaSa,Gc2,Toft15}.)
For example, by \cite{GaSa,Toft15} it follows that if $\mascB$ in Definition
\ref{Def:ModSpaces} is a mixed quasi-normed space of Lebesgue type
and $\phi \in M^r_{(v)}(\rr d)\setminus 0$, then $M(\omega ,\mascB )$ is a
quasi-Banach space. Moreover, $f\in M(\omega ,\mascB )$
if and only if $V_\phi f \cdot \omega \in \mascB$, and different choices of
$\phi$ give rise to equivalent quasi-norms in \eqref{Eq:ModSpNorm}.
We also note that for any such $\mascB$, then
$$
\Sigma _1(\rr d) \subseteq M(\omega ,\mascB ) \subseteq \Sigma _1'(\rr d).
$$

\par

Now let $\phi \in \maclS _{1/2}(\rr d)\setminus 0$, $r\in (0,1]$,
$\omega ,v\in \mascP _E(\rr d)$, $\mascB
\subseteq L^r_{loc}(\rr d)$ be a $v$-invariant quasi-Banach space.
We are especially interested in the modulation spaces $M^\infty
(\omega ,\mascB )$ and $W^\infty (\omega ,\mascB )$, which are
defined as the sets of
all $f\in \maclS _{1/2}'(\rr d)$ such that
\begin{align*}
\nm f{M^\infty (\omega ,\mascB )} &\equiv
\Nm {\left ( \underset {x\in \rr d} \esssup |V_\phi f(x,\cdo )|
\cdot \omega \right )}{\mascB}
\intertext{respective}
\nm f{W^\infty (\omega ,\mascB )} &\equiv
\underset {x\in \rr d} \esssup 
\left (\nm {V_\phi f(x,\cdo )\cdot \omega}{\mascB} \right )
\end{align*}
are finite. By straight-forward computations it follows that
$$
M^\infty (\omega ,\mascB ) \hookrightarrow W^\infty (\omega ,\mascB ).
$$

\par

%
% Definitions of spaces (GS and Gevrey classes)
%

\subsection{Classes of periodic elements}

\par

We shall mainly view three aspects on periodicity. First we consider
spaces of periodic Gevrey functions and their duals. Thereafter we focus
(formal) spaces of Fourier series expansions. Finally we consider
periodic Gelfand-Shilov distributions. In Section \ref{sec2} we show that these
different approaches lead to the same type of spaces.

\par

%Let $E$ be a non-degenerate parallelepiped in $\rd$. Then there is a basis
%$e_1,\dots,e_d$ of $\rd$ such that
%$$
%E = \sets{x_1e_1+\cdots+x_de_d}{(x_1,\dots,x_d)\in\rd,\ 0\leq x_j\leq 1,\
%j=1,\dots,d}.
%$$
%The corresponding lattice, dual parallelepiped and dual lattice are given by
%$$
%\Lambda _E =\sets{n_1e_1+\cdots+n_de_d}{(n_1,\dots,n_d)\in \zz d},
%$$
% 
%$$
%E'=\sets {x_1e'_1+\cdots+x_de'_d}{(x_1,\dots,x_d)\in\rd,\ 0\leq x_j\leq 1,\ j=1,\dots,d},
%$$
%and
%$$
%\Lambda'_E=\Lambda_{E'}=\sets{n_1e'_1+\cdots+n_de'_d}{(n_1,\dots,n_d)
%\in \zz d},
%$$
%respectively, where $e'_1,\dots ,e'_d$ is the basis which satisfies 
%$$
%\scal {e_j} {e'_k}=2\pi \delta_{jk}
%\quad \text{for every}\quad
%j,k =1,\dots, d. 
%$$
%Evidently, there is a matrix $T_E$ with $e_1,\dots ,e_d$ as the image of
%the standard basis. Then the image of the standard basis of $T_{E'}= 2\pi(T^{-1}_E)^t$
%is given $e_1',\dots ,e_d'$.

\par

Let $s,t\in\mathbf{R}_{+}$ be such that $s+t\geq 1$, $f\in (\mathcal S_{t}^{s})'(\rd)$ 
and let $E$ be a non-degenerate parallelepiped in $\rr d$.
Then $f$ is called \emph{$E$-periodic} or \emph{$\Lambda _E$-periodic} if $f(x+j)=f(x)$ 
for every $x\in \rd$ and $j\in \Lambda _E$.

\par

The sets
of periodic elements in $(\maclS _t^s)'(\rr d)$ and $(\Sigma _t^s)'(\rr d)$
are denoted by $(\maclS _t^{E,s})'(\rr d)$ and $(\Sigma _t^{E,s})'(\rr d)$,
respectively.

\par

We note that for any $\Lambda _E$-periodic function $f\in C^{\infty}(\rd)$, we have 
\begin{align}
f &= \sum_{\alpha\in \Lambda ' _E} c(f,{\alpha})e^{i\scal \cdo \alpha},
\label{Eq:Expan2}
\intertext{where $c(f,{\alpha})$ are the Fourier coefficients given by}
c(f,{\alpha}) &\equiv \vert E \vert ^{-1}
(f,e^{i\scal \cdo \alpha})_{L^{2}(E)}.\notag
\end{align}

\par

For any $s\ge 0$ and non-degenerate parallelepiped $E\subseteq \rd$ we let $\maclE _{0,s}^{E}(\rd)$ and 
$\maclE _{s}^{E}(\rd)$ be the sets of all $E$-periodic elements
in $\maclE _{0,s}(\rd)$ and in $\maclE _{s}(\rd)$, respectively. Evidently,
$$
\maclE _s^E(\rr d)\simeq \maclE _s(\rr d/\Lambda _E)
\quad \text{and}\quad
\maclE _{0,s}^E(\rr d)\simeq \maclE _{0,s}(\rr d/\Lambda _E),
$$
which is a common approach in the literature.
The duals of $\maclE _{0,s}^{E}(\rd)$ and $\maclE _{s}^{E}(\rd)$ are denoted by
$(\maclE ^{E}_{0,s})'(\rd)$ and $(\maclE ^{E}_{s})'(\rd)$, respectively.

\par

In Section \ref{sec3} we shall characterise spaces of periodic elements given in the
following definition.

\par

\begin{defn}\label{Def:PerQuasiBanachSpaces}
Let $E\in\rd$ be a non-degenerate parallelepiped, $\omega$ be a weight on $\rd$ and let $\mascB$ be a
quasi-Banach space  continuously embedded in $l^{\infty}_{loc}(\Lambda' _E)$. Then
$\maclE ^E(\omega,\mascB )$ consists of all 
$f\in(\maclE_{0}^E)'(\rd)$ such that 
\begin{equation*}
\nm f{\maclE ^E(\omega,\mascB)}\equiv \nm {\{ c(f,\alpha)\omega(\alpha) \}
_{\alpha \in \Lambda '_E}} {\mascB}
\end{equation*}
is finite.
\end{defn}

\par

Next we introduce suitable spaces of formal Fourier series expansions.
For any $r\in \mathbf{R}$ and $s>0$, we let $\maclG _{s,r}^{E}(\rd)$
be the set of all formal expansions
\begin{equation}\tag*{(\ref{Eq:Expan2})$'$}
f=\sum_{\alpha\in \Lambda '_{E}} c({\alpha})e^{i\scal \cdo \alpha}
\end{equation}
such that 
\begin{equation*}
\nm f{\maclG _{s,r}^{E}}
\equiv
\sup_{\alpha \in \Lambda _{E}' }\vert c(\alpha)
e^{r\vert \alpha\vert^{\frac{1}{s}}}\vert
\end{equation*}
is finite. Then $\maclG _{s,r}^{E}(\rd)$ is a Banach space under the norm
$\Vert \cdo \Vert_{\maclG _{s,r}^{E}}$.
We let 
\begin{equation*}
\begin{alignedat}{2}
\maclG _{s}^{E}(\rd) &= \underset {r>0}\indlim \maclG _{s,r}^{E}(\rd),&
\quad
\maclG _{0,s}^{E}(\rd) &= \underset {r>0} \projlim \maclG _{s,r}^{E}(\rd),
\\[1ex]
(\maclG _{s}^{E})'(\rd) &= \underset {r<0}\projlim \maclG _{s,r}^{E}(\rd),&
\quad
(\maclG _{0,s}^{E})'(\rd) &= \underset {r<0}\indlim\maclG _{s,r}^{E}(\rd).
\end{alignedat}
\end{equation*}
We also let $\maclG _{0,0}^E(\rr d)$ be the set of all constant functions on $\rr d$,
$\maclG _{0}^E(\rr d)$ be the set of all expansions in \eqref{Eq:Expan2}$'$
such that all but finite numbers of $c(\alpha )$ are zero, and we let
$(\maclG _{0}^E)'(\rr d)$ be the set of all formal expansions of the form
\eqref{Eq:Expan2}$'$ (cf. \cite{Zi}).

\par

The topology of $(\maclG _{0}^E)'(\rr d)$ is defined through the semi-norms
$$
\nm f{[N]}\equiv
\sup_{\alpha \in \Lambda _{E}',\, |\alpha |\le N}\vert c(\alpha )\vert ,
\qquad f=\sum_{\alpha\in \Lambda '_{E}} c({\alpha})e^{i\scal \cdo \alpha},
$$
in which $(\maclG _{0}^E)'(\rr d)$ becomes a Fr{\'e}chet space. The set
$\maclG _{0}^E(\rr d)$ is the union of finite-dimensional spaces of trigonometric
polynomials with canonical topologies, and $\maclG _{0}^E(\rr d)$ is equipped
with the inductive limit topology of these vector spaces.

\par

Evidently, if $f\in \maclG _{s}^{E}(\rd)$ or $f\in \maclG _{0,s}^{E}(\rd)$
for some $s\ge 0$ is given by \eqref{Eq:Expan2}, then $\sum
_{\alpha \in \Lambda _{E}'}|c(\alpha )|$
is convergent, and we may identify $f$ by a continuous
$E$-periodic function.

\par

If $s\ge 0$, $f\in(\maclG _{s}^{E})'(\rd)$ and $\phi\in\maclG _{s,r}^{E}(\rd)$ or
$f\in(\maclG _{0,s}^{E})'(\rd)$ and $\phi\in\maclG _{0,s}^{E}(\rd)$,
then we set
\begin{equation*}
(f,\phi)_{E}=\sum_{\alpha \in \Lambda '_{E}} c(f,{\alpha})
\overline{c(\phi ,{\alpha})}
\end{equation*}
and
\begin{equation*}
\scal f\phi_{E} =\sum _{\alpha \in \Lambda _{E}'} c(f,{\alpha}) c(\phi ,{\alpha}),
\end{equation*}
and it follows that the duals of $\maclG _{s}^{E}(\rd)$ and $\maclG _{0,s}^{E}(\rd)$ can be 
identified by $(\maclG _{s}^{E})'(\rd)$ and $(\maclG _{0,s}^{E})'(\rd))$ respectively.
We also note that by the identification of $\maclG _{s}^{E}(\rd)$ as subspace of
$E$-periodic
continuous functions, the form $(\cdo,\cdo)_{E}$ on $\maclG _{s,r}^{E}(\rd)$ 
extends uniquely to a scalar product on $L^{2}(E)$ and that 
\begin{equation*}
\Vert f \Vert_{E}= \vert E\vert ^{-\frac{1}{2}}\Vert f\Vert
_{L^{2}(E)},
\end{equation*}
where $\vert E\vert$ is the volume of $E$.

\par

In Section \ref{sec2} we show that
\begin{alignat}{1}
\mathcal E_s^E(\rd) &=\maclG _s^E(\rd),
\qquad
\mathcal E_{0,s}^E(\rd) = \maclG _{0,s}^E(\rd) ,
\label{Eq:PerGevIdent1}
\\[1ex]
(\mathcal E_s^E)'(\rd) &=(\maclG _s^E)'(\rd) = (\maclS _t^{E,s})'(\rr d),
\label{Eq:PerGevDistIdent1}
\intertext{and}
(\mathcal E_{0,s}^E)'(\rd) &= (\maclG _{0,s}^E)'(\rd) = (\Sigma _t^{E,s})'(\rr d).
\label{Eq:PerGevDistIdent2}
\end{alignat}

\par

\begin{rem}\label{Rem:PerDistr}
We note that if $f\in \mascS '(\rr d)$ is $E$-periodic given by \eqref{Eq:Expan2}
and $\phi \in \mascS (\rr d)$, then
\begin{equation}\label{Eq:PerDistAction}
\scal f\phi = (2\pi )^{\frac d2}\sum _{\alpha \in E}c(f,\alpha )
\widehat \phi (-\alpha ).
\end{equation}
If instead $f\in (\maclG _s^{E})'(\rr d)$, then the map which takes
$\phi \in \maclS _t^s(\rr d)$ into the right-hand side of
\eqref{Eq:PerDistAction}, defines an element in $(\maclS _t^s)'(\rr d)$
since
$$
|c(f,\alpha )|\lesssim e^{r_2|\alpha |^{\frac 1s}}
\quad \text{and}\quad
|\widehat \phi (\xi )|\lesssim e^{-r_1|\xi |^{\frac 1s}}
$$
for every $r_2>0$ and some $r_1>0$.
%$|c(f,\alpha )|\lesssim e^{r_2|\alpha |^{\frac 1s}}$ for every
%$r_2>0$ and $|\widehat \phi (\xi )|\lesssim e^{-r_1|\xi |^{\frac 1s}}$
%for some $r_1>0$.
Similar arguments
hold with $\maclG _{0,s}^{E}$ and $\Sigma _t^s$ in place of
$\maclG _{s}^{E}$ and $\maclS _t^s$ (at each place).

\par

This shows that any $f$ in
$\maclG _{s}^{E}$ (in $\maclG _{0,s}^{E}$) can be identified as an element
in $(\maclS _t^s)'(\rr d)$ (in $(\Sigma _t^s)'(\rr d)$) and that the mappings which take
$\maclG _{s}^{E}(\rr d)$ and $\maclG _{0,s}^{E}(\rr d)$ into
$(\maclS _t^s)'(\rr d)$ and $(\Sigma _t^s)'(\rr d)$, respectively, are
continuous. 
\end{rem}

\par

%%%%%%%%%%%%%%%%%%%%%%%
\section{Characterizations of periodic functions and
distributions}\label{sec2}
%%%%%%%%%%%%%%%%%%%%%%%

\par

In this section we show that \eqref{Eq:PerGevIdent1}--\eqref{Eq:PerGevDistIdent2}
hold.
At the same time we deduce characterizations of such
spaces in terms of suitable estimates on the short-time Fourier transforms
of the involved functions and
distributions. We also deduce a convenient
formula for computing the Fourier coefficients.

\par

%The following proposition shows that \eqref{Eq:PerGevIdent1} holds.
%More precisely, we have the following results.
In the first result we show that \eqref{Eq:PerGevIdent1}--\eqref{Eq:PerGevDistIdent2} hold.

\par

\begin{thm}\label{Thm:Equality}
%Let $E\subseteq \rr d$ be a non-degenerate parallelepiped,
%and let $s,t>0$ be such that $s+t\geq 1$.
%Then \eqref{Eq:PerGevIdent1} and \eqref{Eq:PerGevDistIdent1} hold true.
%If in addition $(s,t)\neq (\frac{1}{2},\frac 12)$, then \eqref{Eq:PerGevDistIdent2}
%holds true.
Let $E\subseteq \rr d$ be a non-degenerate parallelepiped,
and let $s,t>0$ be such that $s+t\geq 1$.
Then the following is true:
\begin{enumerate}
\item if $f\in (\maclG _s^E)'(\rd)$ ($f\in (\maclG _{0,s}^E)'(\rd)$) is given by
\eqref{Eq:Expan2} and $\phi \in \maclS _t^s(\rr d)$ ($\phi \in \Sigma _t^s(\rr d)$),
then \eqref{Eq:PerDistAction} holds;

\vrum

\item the equalities in \eqref{Eq:PerGevIdent1} and \eqref{Eq:PerGevDistIdent1} hold true.
If in addition $(s,t)\neq (\frac{1}{2},\frac 12)$, then \eqref{Eq:PerGevDistIdent2}
holds true.
\end{enumerate}
\end{thm}

\par

In Theorem \ref{Thm:Equality} it is understood that in (2) we interpret
the elements in $(\maclG ^{E}_s)'(\rd)$ ($(\maclG ^{E}_{0,s})'
(\rd)$) as elements in $(\maclS _t^s)'(\rr d)$ ($(\Sigma^{s}_t)'(\rd)$),
which is possible in view of Remark \ref{Rem:PerDistr}.

\par

The next result shows that the form $(\cdo ,\cdo )_{E}$
can be obtained in terms of suitable integrations of short-time Fourier transforms.

\par

\begin{thm}\label{Prop:STFTDescrForm}
Let $E\subseteq \rr d$ be a non-degenerate parallelepiped,
$s,t>0$ be such that $s+t\ge 1$,
$f\in (\maclE _s^{E})'(\rr d)$, $\psi \in \maclE _s^{E}(\rr d)$
and $\phi \in \maclS _t^s(\rr d)\setminus 0$. Then
\begin{align}
(x,\xi ) &\mapsto (V_\phi f)(x,\xi )\overline{(V_\phi \psi)(x,\xi )}\in L^1(E\times \rr d)
\label{Eq:STFTL1Prop}
\intertext{and}
(f,\psi )_{E} &= (\nm \phi {L^2}^2|E|)^{-1}
\int _{E}
\left (
\int _{\rr d} (V_\phi f)(x,\xi )\overline{(V_\phi \psi)(x,\xi )}\, d\xi
\right )
\, dx
\label{Eq:STFTDescrForm}
\end{align}
The same holds true with $\maclE _s^{E}$ and $\Sigma _t^s$
in place of $\maclE _{0,s}^{E}$ and $\maclS _t^s$.
\end{thm}

\par

\begin{rem}\label{Rem:FourCoeff}
Let $f$ and $\phi$ be the same as in Theorem \ref{Prop:STFTDescrForm}.
Then it follows from \eqref{Eq:STFTDescrForm} that 
\begin{multline*}
c(f,\alpha)=(f,e^{i\scal\cdo \alpha})_{E}
\\[1ex]
= (\nm \phi{L^2({\rd})}^2 |E|)^{-1} \int _{E} \left(\int _{\rd} (V_\phi f)(x,\xi)
\widehat{\phi}(\alpha - \xi) e^{-i\scal x {\alpha - \xi}} \,d\xi \right ) \, dx.
\end{multline*}
\end{rem}
%%%%%%%%%%%%%%%%%%%%%%%%%%%

\par

We also have the following characterizations of periodic ultra-distributions
in terms of short-time Fourier transforms, analogous to Propositions
\ref{stftGelfand2} and \ref{stftGelfand2dist}.

\par

\begin{thm}\label{Prop:PerGSDistSTFT2}
Let $E\subseteq \rr d$ be a non-degenerate parallelepiped,
$s> 0$, $f$ be an $E$-periodic
Gevrey distribution on $\rr d$, and let $\phi \in \maclS _t^s(\rr d)\setminus 0$
($\phi \in \Sigma _t^s(\rr d)\setminus 0$) for some $t\ge 0$. Then the following
is true:
\begin{enumerate}
\item $f\in (\maclE _s^{E})'(\rr d)$ ($f\in (\maclE _{0,s}^{E})'(\rr d)$) if and only if
 $|V_\phi f(x,\xi )|\lesssim e^{r|\xi |^{\frac 1s}}$ for every $r>0$
(for some $r>0$).

\vrum

\item $f\in \maclE _s^{E}(\rr d)$ ($f\in \maclE _{0,s}^{E}(\rr d)$) if and only if 
$|V_\phi f(x,\xi )|\lesssim e^{-r|\xi |^{\frac 1s}}$ for some $r>0$
(for every $r>0$).
\end{enumerate}
\end{thm}

\medspace

%When proving Theorem \ref{Thm:Equality} we 

The identities \eqref{Eq:PerGevIdent1} and the first two equalities in
\eqref{Eq:PerGevDistIdent1} and \eqref{Eq:PerGevDistIdent2} in Theorem
\ref{Thm:Equality} also hold for $s=0$, which is a consequence
of the following result.
%are immediate consequences of the following result.

\par

\begin{prop}\label{Prop:ExpIdent1}
Let $s\geq 0$ and let $E\subseteq \rr d$ be a non-degenerate parallelepiped.
Then \eqref{Eq:PerGevIdent1} and the first equalities
in \eqref{Eq:PerGevDistIdent1} and \eqref{Eq:PerGevDistIdent2} hold.
%$\maclG _{s}^{E}(\rd)=\maclE _{s}^{E}(\rd)$ and
%$\maclG _{0,s}^{E}(\rd)=\maclE  _{0,s}^{E}(\rd)$.
In particular the map 
\begin{equation}\label{Eq:1}
f \mapsto \sum_{\alpha\in \Lambda _{E}'} c(f,\alpha) e^{i\scal \cdo \alpha}
\end{equation}
is a homeomorphism from $\maclE _{s}^{E}(\rd)$ to $\maclG _{s}^{E}(\rd)$ and from 
$\maclE _{0,s}^{E}(\rd)$ to $\maclG _{0,s}^{E}(\rd)$, and extend uniquely to 
homeomorphisms from $(\maclE _{s}^{E})'(\rd)$ to $(\maclG _{s}^{E})'(\rd)$ and from 
$(\maclE _{0,s}^{E})'(\rd)$ to $(\maclG _{0,s}^{E})'(\rd)$.
\end{prop}

\par

Proofs of \eqref{Eq:PerGevIdent1} in Theorem \ref{Thm:Equality} in the case $s>0$
can be found in e.{\,}g. \cite{DaRu1,Go,Pil2}. In order to be self-contained we here
present a proof including this part as well.

\par

\begin{proof}
We only prove the first equality in \eqref{Eq:PerGevIdent1}. The second one follows
by similar arguments and is left for the reader. The first equalities in
\eqref{Eq:PerGevDistIdent1} and \eqref{Eq:PerGevDistIdent2} are then immediate
consequences of \eqref{Eq:PerGevIdent1} and duality.
%Let $N>0$ be an integer such that $2N\geq s$.

\par

First we consider the case when $s>0$. Assume that $f\in
\maclE _{s}^{E}(\rd )\setminus 0$ and $c(f,\alpha)=0$ for every $\alpha$.
Then Bessel's equality gives
\begin{equation*}
\sum _{\alpha \in \Lambda _{E}'} \vert c(f,\alpha )\vert^{2}
=
\vert E\vert^{-1}\int_{E}\vert f(x)\vert^{2} dx > 0,\quad
f\in L^{2}(E),
\end{equation*}
and it follows that $c(f,\alpha ) \neq 0$ for at least one $\alpha \in \Lambda _{E}'$.
Hence the right-hand side of \eqref{Eq:1} is non-zero as an element in
$\maclG _s^{E}(\rr d)$, and the injectivity follows.

\par 

Next we show that 
\begin{equation}\label{Eq:1A}
\maclG _{s}^{E}(\rd)\subseteq\maclE _{s}^{E}(\rd)\quad \text{and} \quad
\maclG _{0,s}^{E}(\rd)\subseteq\maclE  _{0,s}^{E}(\rd).
\end{equation}
Suppose that $f$ is given by \eqref{Eq:Expan2}, where  
$\vert c(f, {\alpha})\vert \lesssim e^{-r\vert \alpha\vert^{\frac{1}{s}}}$ when
$\alpha \in \Lambda _E'$, for some
$r>0$. Then $f$ is $E$-periodic and smooth, and
\begin{equation*}
\Vert\partial^{\beta}f\Vert_{L^{\infty}}\lesssim \sup_{\alpha}(\vert \alpha
\vert^{\vert \beta\vert} e^{-r\vert\alpha\vert^{\frac{1}{s}}})
\end{equation*}
for some $r>0$. 
The embeddings $\maclG _{s}^{E}(\rd)\subseteq\maclE _{s}^{E}(\rd)$ and
$\maclG _{0,s}^{E}(\rd)\subseteq\maclE  _{0,s}^{E}(\rd)$ follow if we
prove that 
\begin{equation}\label{Eq:3}
\vert \alpha\vert ^{\vert \beta\vert} e^{-r\vert\alpha\vert ^{\frac{1}{s}}} \lesssim
h^{\vert \beta\vert }\beta !^s,
\end{equation}
for some $h\asymp\frac{1}{r^s}$.
In order to show \eqref{Eq:3} we consider
\begin{equation*}
g(t)=t^k e^{-rt^{\frac{1}{s}}},\qquad t>0.
\end{equation*}
Then
$$
g'(t)=\left ( kt^{k-1} - \frac {
rt^{k+\frac{1}{s}-1}}{s} \right )
e^{-rt^{\frac{1}{s}}}
$$
is equal to $0$, if and only if $t=t_{0}=\left( \frac{ks}{r}\right)^s$ in which $g$
attains its maximum. Hence
\begin{equation*}
0<g(t)\leq g(t_0)= \left( \frac{ks}{r}\right)^{ks} e^{-ks},
\end{equation*}
and Stirling's formula gives 
\begin{equation}\label{Eq:4}
g(t)\lesssim \left(\frac{s}{r}\right)^{sk} k!^s = h_1^k k!^s,
\quad \text{where}\quad
h_1=\left(\frac{s}{r}\right)^s .
\end{equation}
By \eqref{Eq:4} we now get
$$
\vert \alpha\vert ^{\vert\beta\vert} e^{-r|\alpha |^\frac{1}{s}}\lesssim
h_1^{\vert\beta\vert} \vert\beta\vert !^s\leq \left(d^s h_1\right)
^{\vert\beta\vert}\beta! ^s,
$$
which gives \eqref{Eq:3} and thereby \eqref{Eq:1A}.

\par

In order to prove the opposite embedding we let $f\in \maclE _s^{E}(\rd )$.
Since smooth periodic functions agree
with their Fourier series expansions
with absolutely convergent Fourier series, \eqref{Eq:Expan2} holds with
\begin{equation*}
c(f,\alpha ) = |E|^{-1}\int _{E} f(x) e^{-i\scal x\alpha}\, dx.
\end{equation*}
By differentiations we get
\begin{equation*}
\alpha ^{\beta}\vert c(f,\alpha )\vert = \vert c(f^{(\beta)}),\alpha )\vert
\leq C h^{\vert \beta\vert}\beta !^{s},
\end{equation*}
which gives
\begin{equation}\label{Eq:CoeffEst}
\vert c(f,\alpha )\vert \leq C \prod_{j=1}^{d} g_{\alpha _j}(\beta _j)
\end{equation}
where $g_0(t)=1$ and
\begin{equation*}
g_k(t)
=
\frac{h^{t} t^{st}}{k^t}, \qquad t\ge 0,
\end{equation*}
when $k\ge 1$ is an integer. If $k\ge 1$, then $g_k'(t)=0$
exactly for
\begin{equation*}
t =t_0= \frac{k^{\frac{1}{s}}}{h^{\frac{1}{s}}e},
\end{equation*}
in which $g_k$ attains its global maximum.
By straight-forward computations we get
\begin{equation*}
g_k(t)\le g_k(t_0)=e^{-\frac1{h_1} {k^{\frac{1}{s}}}},\qquad h_1\equiv
\frac{h^{\frac{1}{s}}e}{s}.
\end{equation*}

\par

By letting $k=\alpha _j$ and $t=\beta _j$ in the last estimate,
\eqref{Eq:CoeffEst} gives
$$
\vert c(f,\alpha )\vert \leq C \prod_{j=1}^{d} e^{-\frac1{h_1} {\alpha _j^{\frac{1}{s}}}}
\le
Ce^{-\frac1{h_2} | \alpha |^{\frac{1}{s}}},
$$
for some $h_2$ which is proportional to $h_1$.
This shows that equalities hold in \eqref{Eq:1A}, and the result follows.

\par

It remains to consider the case when $s=0$. First assume that
$f\in \maclG _0^E(\rr d)$. Then for some integer $N\ge 0$ we have
\begin{align*}
f(x) &= \sum _{\alpha \in \Lambda _{E,N}'} c(\alpha )e^{i\scal x\alpha},
\intertext{where}
\Lambda _{E,N}' &= \sets {\alpha \in \Lambda _E'}{|\alpha |\le N}
\end{align*}
For every $\beta \in \nn d$ we get
$$
|\partial ^\beta f(x)| \lesssim \max _{\alpha \in \Lambda _{E,N}'} |c(\alpha )\alpha ^\beta |
\lesssim N^{|\beta |},
$$
which implies that $f\in \maclE _0^E(\rr d)$, and we have shown that
$\maclG _0^E(\rr d)\subseteq \maclE _0^E(\rr d)$.

\par

Assume instead that $f\in \maclE _0^E(\rr d)$. Then $f$ is $E$-periodic,
smooth and $|\partial ^\beta f(x)|\lesssim h^{|\beta |}$, for some $h>0$.
This implies that $f$ is given by \eqref{Eq:Expan2}. By differentiations
and Bessel's equality we get
$$
\sum _{\alpha \in \Lambda _E'} |c(f,\alpha )\alpha ^\beta |^2
\asymp \nm {\partial ^\beta f}{L^2(E)}^2 \lesssim h^{2|\beta |}.
$$
This gives
$$
\sup _{\beta \Lambda _E'} \left | c(f,\alpha ) \left ( \frac \alpha h \right )^\beta
\right | <\infty ,
$$
which implies that $c(f,\alpha )=0$ when $|\alpha _j|>h$ for some $j=1,\dots ,d$.
That is, the right-hand side of \eqref{Eq:Expan2} must be a finite sum. Hence
$f\in \maclG _0^E(\rr d)$, and \eqref{Eq:PerGevIdent1} follows.

\par

By \eqref{Eq:PerGevIdent1} it follows that the duals $(\maclE _s^{E})'(\rr d)$ and
$(\maclE _{0,s}^{E})'(\rr d)$ of $\maclE _s^{E}(\rr d)$ and
$\maclE _{0,s}^{E}(\rr d)$ agree with
$(\maclG _s^{E})'(\rr d)$ and $(\maclG _{0,s}^{E})'(\rr d)$
through the forms $(\cdo ,\cdo )_{E}$
and $\scal \cdo \cdo _{E}$.
\end{proof}

\par

\medspace

The following lemma is needed for the proof of the second equalities in
\eqref{Eq:PerGevDistIdent1} and \eqref{Eq:PerGevDistIdent2}.

\par

\begin{lemma}
Let $E\subseteq \rr d$ be a non-degenerate parallelepiped, and
$s,t>0$ be such that $s+t\ge 1$. Then
\begin{align}
(\maclE _{s}^{E})'(\rr d)&\hookrightarrow (\maclS _t^{E,s})'(\rr d),
\label{Eq:PerDistEmb1}
\intertext{and if in addition $(s,t)\neq (\frac 12,\frac 12)$, then}
(\maclE _{0,s}^{E})'(\rr d)&\hookrightarrow (\Sigma _t^{E,s})'(\rr d).
\label{Eq:PerDistEmb2}
\end{align}
\end{lemma}

\par

\begin{proof} 
Let $f\in (\maclE _{s}^{E})'(\rr d)$ ($f\in (\maclE _{0,s}^{E})'(\rr d)$) be
equal to $0$ in $(\maclS _t^s)'(\rr d)$
($(\Sigma _t^s)'(\rr d)$).
Since the continuity is already proved, it remains to show that $f$ is
equal to $0$ in $(\maclE _{s}^{E})'(\rr d)$
($(\maclE _{0,s}^{E})'(\rr d)$).

\par

Let $\alpha _0\in \Lambda _{E}'$ and $\phi _{\ep ,\alpha _0}\in \maclS _t^s(\rr d)$
($\phi _{\ep ,\alpha _0}\in \Sigma _t^s(\rr d)$) be such that
$$
\widehat \phi _{\ep ,\alpha _0} = \widehat \fy (\ep ^{-1}(\xi +\alpha _0)) 
$$
for some function $\fy$ which satisfies $\widehat \fy (0)=1$. Then
$$
0=\scal f{\phi _{\ep ,\alpha _0}} =\sum _{\alpha \in \Lambda _{E}'}
c(f,\alpha )\widehat \fy (\ep ^{-1}(\alpha _0-\alpha )) \to c(f,\alpha _0)
$$
as $\ep \to 0$. Hence $c(f,\alpha )=0$ for every $\alpha$, which shows
that $f=0$ in $(\maclE _{s}^{E})'(\rr d)$ ($(\maclE _{0,s}^{E})'(\rr d)$).
\end{proof}

\par

In order to prove the opposite embeddings to \eqref{Eq:PerDistEmb1}
and \eqref{Eq:PerDistEmb2} we need the following propositions, which
are at the same time main ingredients in the proof of
Theorem \ref{Prop:PerGSDistSTFT2}. They, show that periodic Gelfand-Shilov
distributions and periodic elements in Gevrey classes can be characterized by
suitable estimates of short-time Fourier transforms.

\par

\begin{prop}\label{Prop:PerGevrSTFT}
Let $E\subseteq \rr d$ be a non-degenerate parallelepiped, $s,t\ge 0$
be such that $s+t> 1$, $f$ be a $E$-periodic
Gelfand-Shilov distribution on $\rr d$, and let $\phi \in \maclS _t^s(\rr d)\setminus 0$
($\phi \in \Sigma _t^s(\rr d)\setminus 0$). If $f\in \maclE _s^{E}(\rr d)$
($f\in \maclE _{0,s}^{E}(\rr d)$), then
\begin{equation}\label{Eq:STFT}
|V_\phi f(x,\xi )|\lesssim e^{-r|\xi |^{\frac 1s}},
\end{equation}
for some $r>0$ (for every $r>0$).
\end{prop}

\par

\begin{prop}\label{Prop:PerGSDistSTFT}
Let $E\subseteq \rr d$ be a non-degenerate parallelepiped, $s,t\ge 0$ be such that
$s+t> 1$, $f$ be a $E$-periodic
Gelfand-Shilov distribution on $\rr d$, and let $\phi \in \maclS _t^s(\rr d)\setminus 0$
($\phi \in \Sigma _t^s(\rr d)\setminus 0$). Then the following conditions are
equivalent:
\begin{enumerate}
\item $f\in (\maclS _t^s)'(\rr d)$ ($f\in (\Sigma _t^s)'(\rr d)$);

\vrum

\item $|V_\phi f(x,\xi )|\lesssim e^{r|\xi |^{\frac 1s}}$ for every $r>0$
(for some $r>0$).
\end{enumerate}
\end{prop}

\par

\begin{proof}[Proof of Proposition \ref{Prop:PerGevrSTFT}]
We only prove the assertion in the Roumeu case. The Beurling
case follows by similar arguments and is left for the reader.

\par

For some $r>0$ we have
$$
|\widehat \phi (\xi )|\lesssim e^{-r|\xi |^{\frac 1s}}.
$$
Assume that $f\in \maclE _s^{E}(\rr d)$. % and $\maclS _t^s(\rr d)$.
Then $f$ is given by \eqref{Eq:Expan2}, where
$$
|c(\alpha )|\lesssim e^{-2r|\alpha |^{\frac 1s}}
$$
for some $r>0$ which is independent of $\alpha \in \Lambda _E'$
and $\xi \in \rr d$. Hence, for some $r>0$ we have
\begin{multline*}
|V_\phi f(x,\xi )| \le \sum _{\alpha \in \Lambda _{E}'}|c(\alpha )
V_\phi (e^{i\scal \cdo \alpha })(x,\xi )|
\\[1ex]
=
\sum _{\alpha \in \Lambda _{E}'}|c(\alpha )
\widehat \phi (\xi -\alpha )|
%\\[1ex]
\lesssim
\sum _{\alpha \in \Lambda _{E}'} e^{-r|\alpha |^{\frac 1s}}
e^{-r(|\alpha |^{\frac 1s} +|\xi - \alpha |^{\frac 1s})}
\\[1ex]
\lesssim
\sup _{\eta \in \rr d} ( e^{-r(|\eta |^{\frac 1s} +|\xi - \eta |^{\frac 1s})} )
\le e^{-rc|\xi |^{\frac 1s}},
\end{multline*}
for some $c>0$ which only depends on $s$. This gives the result.
\end{proof}

\par

\begin{proof}[Proof of Proposition \ref{Prop:PerGSDistSTFT}]
Again we only prove the assertion in the Roumeeu case, leaving
the Beurling case for the reader.

\par

Assume that (1) holds. By
Proposition \ref{stftGelfand2dist} we get
$$
|V_\phi f(x,\xi )|\lesssim e^{r(|x|^{\frac 1t}+|\xi |^{\frac 1s})}
$$
for every $r>0$. Since $f$ is $E$-periodic, it follows that the same holds
true for the map $x\mapsto |V_\phi f(x,\xi )|$, and the previous estimate gives
$$
|V_\phi f(x,\xi )|= |V_\phi f(x+j ,\xi )|\lesssim
e^{r(|x+j|^{\frac 1t}+|\xi |^{\frac 1s})}
$$
for every $j \in \Lambda _{E}$.
By taking the infimum over all $j \in \Lambda _{E}$ we get
$$
|V_\phi f(x,\xi )|\lesssim e^{r|\xi |^{\frac 1s}}
$$
for every $r>0$, and (2) follows.

\par

If instead (2) holds, then
$$
|V_\phi f(x,\xi )|\lesssim e^{r(|x|^{\frac 1t}+|\xi |^{\frac 1s})}
$$
for every $r>0$, and Proposition \ref{stftGelfand2dist} shows that
$f\in (\maclS _t^s)'(\rr d)$. This gives the result.
\end{proof}

\par

\begin{rem}\label{Rem:STFTDescrForm}
Let $f$ and $\psi$ be the same as in Theorem \ref{Prop:STFTDescrForm}.
Then \eqref{Eq:STFTL1Prop} holds in view of Propositions
\ref{Prop:PerGSDistSTFT} and \ref{Prop:PerGevrSTFT}, which
implies that the right-hand side of \eqref{Eq:STFTDescrForm}
makes sense.

\par

From these properties and \eqref{Eq:GSFtransfChar} it
follows that
$$
\sum _{\alpha ,\beta \in \Lambda _{E}'}
|c(f,\alpha )c(\psi ,\beta )\widehat \phi (\xi -\alpha )\widehat \phi (\xi -\beta )|<\infty
$$
for every fixed $\xi \in \rr d$.
\end{rem}

\par

\begin{proof}[Proof of Theorem \ref{Prop:STFTDescrForm}]
Again we only prove the result in the Roumieu case, leaving the Beurling case
for the reader.

\par

By straight-forward computations we get
$$
(V_\phi f)(x,\xi ) = e^{-i\scal x\xi}\sum _{\alpha \in \Lambda _{E}'}
c(f,\alpha )\overline{\widehat \phi (\alpha -\xi )}e^{i\scal x\alpha}.
$$
%\begin{align*}
%f &= \sum 
%\end{align*}
%%
Hence Remark \ref{Rem:STFTDescrForm}, and
Weierstrass and Fubbini's theorems give
\begin{multline*}
\int _{E}
\left (
\int _{\rr d} (V_\phi f)(x,\xi )\overline{(V_\phi \psi)(x,\xi )}\, d\xi
\right )
\, dx
\\[1ex]
=
\int _{\rr d}
\left (
\int _{E} 
\sum _{\alpha ,\beta \in \Lambda _{E}'}
c(f,\alpha )\overline {c(\psi ,\beta )}\, \overline{\widehat \phi (\alpha -\xi )}
{\widehat \phi (\beta -\xi )}e^{i\scal x{\alpha -\beta}}
\, dx
\right )
\, d\xi
\\[1ex]
=
\int _{\rr d}
\left ( 
\sum _{\alpha ,\beta \in \Lambda _{E}'}
c(f,\alpha )\overline {c(\psi ,\beta )}
\, \overline{\widehat \phi (\alpha -\xi )}
{\widehat \phi (\beta -\xi )}\int _{E}e^{i\scal x{\alpha -\beta}}
\, dx
\right )
\, d\xi
\\[1ex]
=
|E|
\int _{\rr d}
\left ( 
\sum _{\alpha \in \Lambda _{E}'}
c(f,\alpha )\overline {c(\psi ,\alpha )}
|\widehat \phi (\alpha -\xi )|^2
\right )
\, d\xi
\end{multline*}

\par

Since
$$
\sum _{\alpha \in \Lambda _{E}'}
|c(f,\alpha )\overline {c(\psi ,\alpha )}|
\int _{\rr d}|\widehat \phi (\alpha -\xi )|^2\, d\xi
=
\nm \phi {L^2}^2\sum _{\alpha \in \Lambda _{E}'}
|c(f,\alpha )\overline {c(\psi ,\alpha )}|
<\infty ,
$$
an other application of Weierstrass theorem now gives
\begin{multline*}
\int _{\rr d}
\left ( 
\sum _{\alpha \in \Lambda _{E}'}
c(f,\alpha )\overline {c(\psi ,\alpha )}
|\widehat \phi (\alpha -\xi )|^2
\right )
\, d\xi
\\[1ex]
=
\sum _{\alpha  \in \Lambda _{E}'}
c(f,\alpha )\overline {c(\psi ,\alpha )}
\int _{\rr d} |\widehat \phi (\alpha -\xi )|^2
\, d\xi
\\[1ex]
=
\nm \phi {L^2}^2\sum _{\alpha \in \Lambda _{E}'}
c(f,\alpha )\overline {c(\psi ,\alpha )}
=
\nm \phi {L^2}^2 (f,\psi )_{E},
\end{multline*}
and the result follows by combining these equalities.
\end{proof}

\par

In the following definition
we assign any element in
$(\maclS^{E,s}_t)'(\rd)$ ($(\Sigma^{E,s}_t)'(\rd)$), an element in
$(\maclE^{E}_s)'(\rd)$ ($(\maclE^{E}_{0,s})'(\rd)$). Then we prove that
the latter element agrees with the former one as element in
$(\maclS^{E,s}_t)'(\rd)$ ($(\Sigma^{E,s}_t)'(\rd)$), which will give
the last part of Theorem \ref{Thm:Equality}.
%%%%%%%%%%%%%%%%%%%%

\par

\begin{defn}\label{Def:FourSer}
Let $E\subseteq \rr d$ be a non-degenerate parallelepiped,
$s,t>0$ be such that $s+t\geq 1$,
$f\in (\maclS_t^{E,s})'(\rd)$ $(f\in(\Sigma_t^{E,s})'(\rd))$, and
let $\phi\in \maclS_t^s(\rd)\setminus 0$ $(\phi\in\Sigma_t^s(\rd )
\setminus 0)$.
\begin{enumerate}
\item The Fourier coefficient $c(f,\alpha)$ for $f$ of order $\alpha\in
\Lambda _{E}'$ is given by \eqref{Eq:STFTDescrForm};

\vrum

\item The Fourier series of $f$ with respect to $E$ is given by
$$
\Fo{E}(f)\equiv \underset{\alpha\in \Lambda _{E}'}\sum c(f,\alpha)
e^{i\scal \cdo \alpha}.
$$
\end{enumerate}
\end{defn}

\par

%%%%%%%%%%%%%%%%%%%%%
Evidently by \eqref{Eq:GSFtransfChar} and
Proposition \ref{Prop:PerGSDistSTFT}
and definitions it
follows that $c(f,\alpha)$ in Definition is well-defined. Hence $\Fo{E}(f)$
in Definition \ref{Def:FourSer} exists as an element in
$(\maclE_0^{E})'(\rd)$.

\par 

Theorem \ref{Thm:Equality} is an immediate consequence of
\eqref{Eq:PerDistEmb1}, \eqref{Eq:PerDistEmb2} and the
following result.

\par

\begin{prop}\label{Prop:Equality}
Let $E\subseteq \rr d$ be a non-degenerate parallelepiped,
$s,t>0$ be such that $s+t\geq 1$, and let $f\in
(\maclS_t^{E,s})'(\rd)$ $(f\in (\Sigma_t^{E,s}) '(\rd))$. Then the
following is true:
\begin{enumerate}
\item $\Fo{E}(f)\in (\maclE_s^{E})'(\rd)$ $(\Fo{E}(f)\in
(\maclE_{0,s}^{E})'(\rd))$.

\vrum

\item $\Fo{E}(f)=f$ as elements in $(\maclS_t^s)'(\rd)$ $((\Sigma_t^s)'(\rd))$.
\end{enumerate}
\end{prop}

\par

We need some preparations for the proof. First we recall that the usual
properties on tensor products also hold for Gelfand-Shilov distributions. More
precisely, the following result follows by similar arguments as the proof of
\cite[Theorem 5.1.1]{Ho1}. The details are left for the reader.

\par

\begin{lemma}\label{Lem:GSTensors}
Let $s,t>0$ and $f_j\in (\maclS _t^s)'(\rr {d_j})$ ($f_j\in (\Sigma _t^s)'(\rr {d_j})$), $j=1,2$.
Then there is a unique $f\in (\maclS _t^s)'(\rr {d_1+d_2})$ ($f\in (\Sigma _t^s)'(\rr {d_1+d_2})$)
such that
$$
\scal f{\phi _1\otimes \phi _2} = \scal {f_1}{\phi _1}\scal {f_2}{\phi _2}
\quad \text{when}\quad
\phi _j\in \maclS _t^s(\rr {d_j})\ 
\text (
\phi _j\in \Sigma _t^s(\rr {d_j})
\text ),\ j=1,2.
$$
If $\phi \in \maclS _t^s(\rr {d_1+d_2})$ ($\phi \in \Sigma _t^s(\rr {d_1+d_2})$),
$$
\psi _1(x_1) = \scal {f_2}{\phi (x_1,\cdo )}
\quad \text{and}\quad
\psi _2(x_2) = \scal {f_1}{\phi (\cdo ,x_2 )},
$$
then $\psi _j\in \maclS _t^s(\rr {d_j})$
($\psi _j\in \Sigma _t^s(\rr {d_j})$), $j=1,2$, and
$$
\scal f\phi = \scal {f_1}{\psi _1} = \scal {f_2}{\psi _2}.
$$
\end{lemma}

\par

We recall that $f$ in Lemma \ref{Lem:GSTensors} is called the tensor product
of $f_1$ and $f_2$ and is usually denoted by $f_1\otimes f_2$.

\par

We also have the following.

\par

\begin{lemma}\label{Lem:EstGSSums}
Let $s,t>0$, $E\subseteq \rr d$ be a non-degenerate parallelepiped, $f\in (\maclS _t^s)'(\rr d)$
($f\in (\Sigma _t^s)'(\rr d)$) and let $\phi _{0},\psi \in
\maclS_t^s(\rd)$ ($\phi_{0},\ \psi\in \Sigma_t^s(\rd)$). Then the following is true:
\begin{enumerate}
\item $\underset{x\in \rd}\sup
\Big ( \sum _{k\in \Lambda_E} |\scal {f }{\phi_{0}(\cdo-x+k)\psi} | \Big )<\infty $.

\vrum

\item $\sum _{k\in \Lambda_E}\phi_{0}(\cdo+k)\psi$ converges in
$\maclS_t^s(\rd )$ ($\Sigma_t^s(\rd)$).
\end{enumerate}
\end{lemma}

\par

\begin{proof}
We only prove the result in the Roumieu case, leaving the Beurling case for the reader.

\par

Let $\phi = \overline {\phi _0}$. We have 
$$
|\phi (y-x+k)| \lesssim e^{-r_{0} |y-x+k|^{\frac{1}{t}}}
\quad \text{and}\quad
|\psi(y)| \lesssim e^{-r_{0}|y|^{\frac{1}{t}}}
$$
for some $r_{0}>0$, which gives 
\begin{equation}\label{Eq:ExpEst1}
|\phi _0(y-x+k)\psi(y)|
\lesssim
e^{-r_{0} ( |y-x+k|^{\frac{1}{t}} + |y|^{\frac{1}{t}} ) }
\lesssim
e^{-cr_{0}(|x-k|^{\frac 1 t}+|y|^{\frac{1}{t}})}
\end{equation}
for some $c\in (0,1)$ which only depends on $t$. Moreover, 
\begin{multline}\label{Eq:ExpEst2}
|\mascF (\phi _0(\cdo -x+k)\psi (y))(\eta)| = |(V_\phi \psi )(x-k,\eta )|
\\[1ex]
\lesssim
e^{-r_{0} ( |x-k|^{\frac 1t} + |\eta|^{\frac 1s} )},
\end{multline}
in view of Proposition \ref{stftGelfand2}.

\par

Since the topology of $\maclS _t^s(\rr d)$ can be obtained through the
semi-norms
$$
\phi \mapsto \nm {\phi e^{r|\cdo |^{\frac 1t}}}{L^\infty}
+
\nm {\widehat \phi e^{r|\cdo |^{\frac 1s}}}{L^\infty},
\qquad r>0,
$$
\eqref{Eq:ExpEst1} and \eqref{Eq:ExpEst2} give
\begin{multline*}
|\scal {f}{\phi_0(\cdo -x+k)\psi}|
\\[1ex]
\lesssim \nm {\phi_0 (\cdo -x+k)\psi e^{cr_0 |\cdo|^{\frac 1 t}/2}}{L^{\infty}}
+\nm {\mathcal{F}{(\phi_0(\cdo -x+k))e^{cr_0|\cdo|^{{\frac 1 s}}/2}}}{L^{\infty}}
\\[1ex]
\lesssim e^{-cr_0|x-k|^{\frac 1t}}.
\end{multline*}
Hence,
%By summing up over all $k \in \Lambda _{E}$, we obtain 
$$
\sup _{x\in \rr d}\left ( \sum _{k\in \Lambda _E} |\scal {f} {\phi_{0}(\cdo -x+k)\psi}| \right )
\lesssim
\sup _{x\in \rr d} \left (\sum _{k\in \Lambda _E} e^{-cr_0|x-k|^{\frac 1t}} \right )< \infty ,
%\leq C,
$$ 
%for some constant $C$ which is independent of $x\in \rd$.
%This gives (1).
and (1) follows.

\par

(2) It is clear that $\psi_{0}=\sum _{k\in\Lambda_{E}} \phi_0 (\cdo +k)\psi$ 
is well-defined and smooth. Let
$$
\Lambda _{E,N}=\sets{k\in \Lambda _{E}}{|k|<N},
\qquad
\Omega _{E,N}=\Lambda_E\setminus \Lambda_{E,N},
$$
and let 
$$
\psi _N =\psi_0 -\sum _{k\in\Lambda _{E,N}} \phi _0(\cdo+k)\psi .
$$
We shall prove that $\psi _N \to 0$ in $\maclS_t^s(\rd )$, as $N$ tends to $\infty$. 

\par

For some $r_0>0$ we have 
$$
|\psi_N |e^{r_0 |\cdo|^{{\frac 1 t}}/2} \leq \sum _{k\in \Omega _{E,N}}
|\phi_0(\cdo +k)||\psi|e^{r_0 |\cdo|^{{\frac 1 t}}/2}
$$
which tends to $0$ as $N$ tends to $\infty$. Moreover,
\begin{equation*}
|\widehat{\psi}_N(\eta)|\lesssim
\sum _{k\in \Omega _{E,N}} |V_\phi \psi(-k,\eta)|
\lesssim \sum _{k\in \Omega _{E,N}} e^{-r_0(|k|^{\frac 1 t}+|\eta|^{\frac 1 s})},
\end{equation*}
by Proposition \ref{stftGelfand2}. Hence
\begin{equation*}
|\widehat{\psi}_N(\eta)e^{r_0 |\eta|^{{\frac 1 s}}/2}|\lesssim
\sum _{k\in \Omega _{E,N}}e^{-r_0(|k|^{\frac 1 t}+|\eta|^{\frac 1 s})/2}
%\lesssim e^{-r_0|\eta|^{\frac 1 s}/2}\sum _{k\in \Omega _{E,N}} 
%e^{-r_0|k|^{\frac 1 t}/2}\to 0
\end{equation*}
which tends to $0$ as $N$ tends to $\infty$. %when $N\to \infty$. Hence 
Consequently,
$$
%\underset{N\to \infty}\lim
\lim _{N\to \infty}
\left(\nm {\psi_N e^{r_0|\cdo|^{\frac 1 t}}}{L^{\infty}} +
\nm {\widehat{\psi}_N e^{r_0|\cdo|^{\frac 1 s}}}{L^{\infty}}\right) =0,
$$
for some $r_0>0$, which shows that $\psi_N\to 0$ in $\maclS _t^s(\rd)$ as
$N\to \infty$ and the result follows. 
\end{proof}

\par

\begin{proof}[Proof of Proposition \ref{Prop:Equality}]
We only prove the result in the Roumieu case. The Beurling case follows
by similar arguments and is left for the reader.

\par

(1) We have 
$$
| \xi+\eta | ^{\frac{1}{s}}\leq c (| \xi | ^{\frac{1}{s}} + | \eta | ^{\frac{1}{s}}),
\quad
| \widehat{\phi}(\xi)| \lesssim e^{-r_{0}| \xi |^{{\frac{1}{s}}}}
\quad \text{and}\quad
| V_\phi f(x,\xi)| \lesssim e^{r| \xi |^{{\frac{1}{s}}}},
$$
for some $c\ge 1$ which only depends on $s$, for some $r_{0}>0$,
and for every $r>0$.
By choosing $r<(c+1)^{-1}r_0$ we get 
\begin{multline*}
\vert c(f,\alpha) | \lesssim \int _{E}
\left (
\int_{\rd} |V_\phi f(x,\xi)| 
|\widehat{\phi}(\alpha - \xi)|\, d\xi 
\right )
\, dx
\\[1ex]
\lesssim 
\int_{E}
\left (
\int_{\rd} e^{r|\xi|^{\frac{1}{s}}} e^{-r_{0}|\alpha -\xi|^{\frac{1}{s}}}
\, d\xi
\right )
\, dx
\\[1ex]
\asymp 
\int_{\rd} e^{r(|\xi|^{\frac{1}{s}} - c|\alpha - \xi|^{\frac{1}{s}})}
e^{-r|\alpha - \xi|^{\frac{1}{s}}}\, d\xi
\\[1ex]
\leq 
e^{rc|\alpha|^{\frac{1}{s}}}\int_{\rd} e^{-r|\alpha - \xi|^{\frac{1}{s}}}\, d\xi 
%\\[1ex]
\asymp 
e^{rc|\alpha|^{\frac{1}{s}}}.
\end{multline*}
Since $rc>0$ can be made arbitrary close to $0$ we get $|c(f,\alpha)|\lesssim
e^{r|\alpha|^{\frac{1}{s}}}$ for every $r>0$, and (1) follows.
%%%%%%%%%%%%%%

\par

(2) Let $f_0=\Fo{E}(f)$, $\phi \in \maclS _{t}^{s}(\rd )\setminus 0$ and
let $\psi \in \maclS_{t}^{s}(\rd)$. Then
\begin{multline}\label{Eq:FourCoeffEst1}
(2\pi )^{-\frac d2}
\nm \phi{L^2(\rd)} |E| \scal {f_{0}}\psi = \sum _{\alpha \in \Lambda _{E}'}
c(\alpha,f) \widehat{\psi}(-\alpha)
\\[1ex]
=\sum _{\alpha \in \Lambda _{E}'} \int _{E}
\left (\int  _{\rd}
(V_{\phi}f)(x,\xi) \widehat{\phi}(\alpha - \xi)
e^{-i \scal x{\alpha-\xi}}\, d\xi\widehat {\psi}(-\alpha)
\right ) \, dx
\\[1ex]
=\int_{E}\left ( \int _{\rd} (V_{\phi}f)(x,\xi) F(x,\xi)\, d\xi
 \right )\, dx,
\end{multline}
where
$$
F(x,\xi)= \sum _{\alpha\in\Lambda _{E}'} \widehat{\phi}(\alpha-\xi)
\widehat{\psi}(-\alpha) e^{-i \scal x{\alpha-\xi}}.
$$
In the last equality in \eqref{Eq:FourCoeffEst1} we have used the fact
\begin{equation*}
\sum _{\alpha\in\Lambda _{E}'} \int _{E} \left(\int _{\rd} 
|(V_{\phi}f)(x,\xi )| |\widehat{\phi} (\alpha - \xi)|\, d\xi \right) \, dx |\widehat {\phi}(-\alpha)|
<\infty,
\end{equation*}
which implies that we may interchange orders of summations and integrations.

\par

We shall rewrite $F(x,\xi)$. By straight-forward computations we get 
$$
\widehat{\phi}(\alpha -\xi ) \widehat{\phi}(-\alpha)
=
(2\pi)^{-\frac{d}{2}}
\mascF ((\phi e^{i \scal \cdo \xi }) * \check \psi )(\alpha ),
$$
and Poisson's summation formula gives
\begin{multline}\label{Eq:FourCoeffEst2}
F(x,\xi)
=
(2\pi)^{-\frac{d}{2}}e^{i\scal x\xi} 
\sum _{\alpha \in \Lambda _{E}'}
\mascF ((\phi e^{i\scal \cdo \xi}) * \check \psi )(\alpha)
e^{-i\scal x\alpha}
\\[1ex]
=
(2\pi)^{-\frac{d}{2}} |E| e^{i\scal x\xi}
\sum _{k\in \Lambda _{E}}
((\phi e^{i\scal \cdo \xi}) * \check \psi)(k - x)
\\[1ex]
=
(2\pi)^{-\frac{d}{2}} |E| \sum _{k \in \Lambda _{E}}
\overline{V_{\phi}(\overline{\psi(\cdo - k)})(x,\xi)}.
\end{multline}
A combination of \eqref{Eq:FourCoeffEst1} and \eqref{Eq:FourCoeffEst2}
leads to
\begin{equation*}
\nm \phi{L^2(\rd)}^{2}\scal {f_{0}}\psi 
=
\left ( \sum _{k \in \Lambda _E} \int _{\rd} V_{\phi}f(x,\xi)
\overline{(V_{\phi}(\overline{\psi(\cdo - k)}))(x,\xi)}\, d\xi
\right )
\, dx.
\end{equation*}
By Fourier inversion Formula we get 
$$
\int _{\rd} (V_\phi f) (x,\xi)\overline{(V_\phi (\overline{\psi (\cdo - k)})(x,\xi)}
\, d\xi = \scal f {|\phi(\cdo - x+k)|^2 \psi}.
$$
Hence 
\begin{equation*}
\nm \phi{L^2(\rd)}^{2} \scal {f_0} \psi
=
\int _{E} 
\sum _{k \in \Lambda _E} \scal f{|\phi(\cdo-x+k)|^2\psi }\, dx,
\end{equation*}
and we shall use Lemma \ref{Lem:GSTensors} and \ref{Lem:EstGSSums}
to reformulate the right-hand side.

\par

By Lemma \ref{Lem:EstGSSums} (1) and Lebesgue's theorem we have
$$
\int _{E} 
\sum _{k \in \Lambda _E} \scal f{|\phi(\cdo-x+k)|^2\psi }\, dx
=
\sum _{k \in \Lambda _E} \int _{E} \scal f{|\phi(\cdo-x+k)|^2\psi }\, dx,
$$
and letting
$$
\psi _k(x,y)=|\phi (y-x+k)|^2\psi (y)\in \maclS _t^s(\rr {2d}),
$$
it follows from Lemma \ref{Lem:GSTensors} that
$$
\int _{E} \scal f{|\phi(\cdo-x+k)|^2\psi }\, dx = \scal {\chi _E\otimes f}{\psi _k}
= \scal f{\phi _0(\cdo +k)\psi},
$$
where
$$
\phi _0 (y) = \int _{E}|\phi (y-x)|^2\, dx \in \maclS _t^s(\rr d).
$$
Here $\chi _E$ is the characteristic function of $E$.

\par

Let $\Lambda _{E,N}$ be the same as in Lemma \ref{Lem:EstGSSums}.
A combining the identities above and Lemma \ref{Lem:EstGSSums}
gives
\begin{multline*}
\nm \phi{L^2(\rd)}^{2}\scal {f_{0}}\psi 
=
\lim _{N\to \infty}\sum _{k \in \Lambda _{E,N}} \scal f{\phi _0(\cdo +k)\psi}
\\[1ex]
=
\lim _{N\to \infty}
\left \langle
f ,
\left ( \sum _{k \in \Lambda _{E,N}}
\phi (\cdo +k) \right )\psi
\right  \rangle
=
\left \langle
f ,
\left ( \sum _{k \in \Lambda _{E}}
\phi (\cdo +k) \right )\psi
\right  \rangle
\\[1ex]
=
\nm \phi{L^2(\rd)}^{2} \scal {f} \psi ,
\end{multline*}
and the result, and thereby Theorem \ref{Thm:Equality} follow.
\end{proof}

\par

\begin{proof}[Proof of Theorem \ref{Prop:PerGSDistSTFT2}]
The assertion (1) and one part of (2) are immediate consequences of
Propositions \ref{Prop:PerGSDistSTFT} and \ref{Prop:PerGevrSTFT}
and Theorem \ref{Thm:Equality}. We need to show that \eqref{Eq:STFT}
for some $r>0$ (every $r>0$) 
is sufficient that $f\in \maclE_s^{E}(\rd)$ ($f\in \maclE_{0,s}^{E}(\rr d)$).

\par

We only consider the Roumieu case. The Beurling case follows
by similar arguments and is left for the reader.

\par

Suppose that $|V_\phi f(x,\xi )|
\lesssim e^{-2r|\xi|^{\frac{1}{s}}}$ for some $r>0$. Then $f\in(\maclE ^{E}_s)'(\rd)$
due to (1), and hence has a Fourier series expansion with coefficients $c(f,\alpha)$.
By Remark \ref{Rem:FourCoeff} and the fact that $\phi\in \maclS^s_t(\rd)$ we have 
\begin{multline*}
|c(f,\alpha)|=|(f,e^{i\scal \cdo \alpha})_{E}|
\\[1ex]
\leq 
(\nm \phi {L^2(\rd)}^2 |E|)^{-1} \int _{E}\left(
\int _{\rd} |V_\phi f(x,\xi)| |\widehat{\phi}(\alpha - \xi)| \, d\xi
\right)\, dx
\\[1ex]
\lesssim \int _{\rd} e^{-r|\xi |^{\frac{1}{s}}} e^{-r(|\xi|^{\frac{1}{s}}
+|\alpha - \xi|^{\frac{1}{s}})} \, d\xi
\\[1ex]
\lesssim e^{-rc|\alpha |^{\frac{1}{s}}}\int _{\rd} e^{-r|\xi |^{\frac{1}{s}}} \, d\xi
\asymp e^{-rc | \alpha |^{\frac{1}{s}}},
\end{multline*}
for some $c>0$ which only depends on $s>0$. This implies that $f\in\maclG_s^{E}(\rd)
=\maclE_s^{E}(\rd)$, and the result follows.
\end{proof}

\par

\begin{rem}
Evidently Theorem \ref{Thm:Equality} (1) is true also when $\phi \in \mascS (\rr d)$
and $f\in (\maclE _s^{E})'(\rr d)\bigcap \mascS '(\rr d)$.

It also follows from the proof of
Theorem \ref{Prop:STFTDescrForm} that the conclusions of that theorem
is also true when in the case when $f\in (\maclE _s^{E})'(\rr d)\bigcap \mascS '(\rr d)$,
$\psi \in (\maclE _s^{E})'(\rr d)\bigcap C^\infty (\rr d)$
and $\phi \in \mascS (\rr d)$. The details are left for the reader.
\end{rem}

\par

\begin{rem}
Let $f$ be a Gelfand-Shilov distribution on $\rr d$ and
let $\phi \in \mascS (\rd )\setminus 0$. Then the following conditions
are equivalent:
\begin{enumerate}
\item $f\in \mascS (\rr d)$;

\vrum

\item $|f(x)|\lesssim \eabs x^{-N}$ and $|\widehat f(\xi )|\lesssim
\eabs \xi ^{-N}$ for every $N\ge 0$;

\vrum

\item $|V_\phi f(x,\xi )|\lesssim \eabs {(x,\xi )}^{-N}$ for every $N\ge 0$.
\end{enumerate}
We also have that $f\in \mascS '(\rr d)$ if and only if
$$
|V_\phi f(x,\xi )|\lesssim \eabs {(x,\xi )}^{N}
$$
for some $N\ge 0$ (cf. e.{\,}g. \cite{ChuChuKim, Gc2} or Remark 1.3 in
\cite{Toft18}).

\par

Now assume that $f$ is an $E$-periodic Gelfand-Shilov distribution on $\rr d$.
From the previous characterizations it follows by similar arguments
as for the proof of Theorem \ref{Prop:PerGSDistSTFT2} that the
following is true:
\begin{enumerate}
\item $f\in \mascS '(\rd )$ if and only if
\begin{equation*}
|V_{\phi}f (x,\xi)| \lesssim \eabs \xi ^{N}
\end{equation*}
for some $N\geq 0$.

\vrum

\item $f\in C^{\infty} (\rd)$ if and only if
\begin{equation*}
|V_{\phi}f (x,\xi)| \lesssim \eabs \xi ^{-N}
\end{equation*}
for every $N\geq 0$.
\end{enumerate}
\end{rem}

\par

\begin{example}\label{Ex:WellPosedHeatEq}
Suppose $E\subseteq \rr d$ is a rectangle, and that $f(t,x)$, $(t,x)\in \mathbf R \times E$
satisfies the heat equation
\begin{equation}\label{Eq:HeatProblem}
\begin{aligned}
\partial _tf &= \Delta _xf,\qquad f(0,x) = f_0(x),\quad
\text{where}
\\[1ex]
f_0(x) &= \sum _{\alpha \in \Lambda _E'} c(f_0,\alpha )e^{i\scal x\alpha}
\end{aligned}
\end{equation}
is a fixed element in $(\maclE _0^E)'(\rr d)$. We are interested to find
well-posedness properties in the framework of the spaces
$\maclE _{s}^E(\rr d)$ and
$(\maclE _{s}^E)'(\rr d)$ when $s\ge 0$, and the spaces $\maclE _{0,s}^E(\rr d)$ and
$(\maclE _{0,s}^E)'(\rr d)$ when $s>0$.

\medspace

The formal solution is given by
$$
f(t,x) = \sum _{\alpha \in \Lambda _E'} c(\alpha )e^{-|\alpha |^2t}
e^{i\scal x\alpha}.
$$
By Theorem \ref{Thm:Equality} it follows that the following is true:
\begin{enumerate}
%\item If $s< \frac 12$ and $f_0\in \maclE _s^E(\rr d)$, then the map
%$t\mapsto f(t,\cdo )$ is a smooth map from $\mathbf R$ to $\maclE _s^E(\rr d)$.
%The same holds true with $(\maclE _s^E)'(\rr d)$ in place of $\maclE _s^E(\rr d)$;
%
%\vrum
%
\item If $0\le s< \frac 12$, then the map
$(t,f_0)\mapsto f(t,\cdo )$ is continuous from $\mathbf R\times \maclE _{s}^E(\rr d)$
to  $\maclE _{s}^E(\rr d)$. Moreover, if $f_0\in \maclE _{s}^E(\rr d)$, then
$t\mapsto f(t,\cdo )$ is smooth from $\mathbf R$ to $\maclE _{s}^E(\rr d)$.
The same holds true with $(\maclE _{s}^E)'(\rr d)$ in place of $\maclE _{s}^E(\rr d)$
at each occurrence;

\vrum

\item If $0<s\le \frac 12$, then the map
$(t,f_0)\mapsto f(t,\cdo )$ is continuous from $\mathbf R\times \maclE _{0,s}^E(\rr d)$
to  $\maclE _{0,s}^E(\rr d)$. Moreover, if $f_0\in \maclE _{0,s}^E(\rr d)$, then
$t\mapsto f(t,\cdo )$ is smooth from $\mathbf R$ to $\maclE _{0,s}^E(\rr d)$.
The same holds true with $(\maclE _{0,s}^E)'(\rr d)$ in place of $\maclE _{0,s}^E(\rr d)$
at each occurrence;

\item If $s= \frac 12$ and $f_0\in (\maclE _s^E)'(\rr d)$, then
$f(t,\cdo )\in \maclE _s^E(\rr d)$ when $t>0$, and
$f(t,\cdo )\in (\maclE _{0,s}^E)'(\rr d)$ when $t<0$.
%$C^\infty (\mathbf R_-; (\maclE _{0,s}^E)'(\rr d))$.
\end{enumerate}
\end{example}

\par

%%%%%%%%%%%%%%%%%%%%%%%
\section{Periodic elements in modulation spaces}\label{sec3}
%%%%%%%%%%%%%%%%%%%%%%%

\par

In this section we show that $E$-periodic elements in the modulation
spaces $M^{\infty}(\omega,\mascB )$ and $W^{\infty}(\omega,\mascB )$
agree with $\maclE^{E}(\omega,\mascB )$ in Definition
\ref{Def:PerQuasiBanachSpaces},
for suitable $\omega$ and $\mascB$.

\par

More precisely we have following extension of \cite[Proposition 2.6]{Re}.

\par

\begin{thm}\label{Thm:PerMod}
Let $E\subseteq \rr d$ be a non-degenerate parallelepiped, $\mathcal O_0$
be as in Definition \ref{Def:MixedLebSpaces}, $\mabfq \in (0,\infty ]^d$,
$\tau \in \operatorname{S_d}$, and let $\omega\in \mascP _E(\rr d)$. Also let
$$
\mascB = L^{\mabfq }_{\mathcal O_0,\tau}(\rr d)
\quad \text{and}\quad
\mascB _0 = \ell ^{\mabfq }_{\mathcal O_0,\tau}(\Lambda _E')
$$
Then
\begin{equation*}
\maclE^{E}(\omega, \mascB _0)
=M^{\infty}(\omega, \mascB )\bigcap (\maclE_{0}^E)'(\rr d)
=W^{\infty}(\omega,\mascB )\bigcap (\maclE_{0}^E)'(\rd).
\end{equation*}
\end{thm}

\par

We note that compactly supported as well as periodic elements in modulation spaces
have been investigated in different contexts. For example, Theorem \ref{Thm:PerMod}
is related to \cite[Proposition 5.1]{RuSuToTo}.

\par

\begin{proof}
%Let $r\in(0,1]$ be such that $\mascB \subseteq L_{loc}^{r}(\rd)$.
It is clear from
the definitions that $M^{\infty}(\omega,\mascB )
\subseteq W^{\infty}(\omega,\mascB )$ (see also \cite{Toft15}). Hence it
sufficies to prove
\begin{equation}\label{Embeddings}
W^{\infty}(\omega,\mascB )\bigcap (\maclE_{0}^E)'(\rd)\subseteq 
\maclE^{E}(\omega,\mascB _0)
\subseteq
M^{\infty}(\omega,\mascB )\bigcap (\maclE_{0}^E)'(\rd).
\end{equation}

\par

For every $f\in W^\infty (\omega ,\mascB  ) \bigcap (\maclE_{0}^E)'(\rr d)$, we have
%$f\in\Omega'_1(\rd)$, which implies that
%$f\in M^{\infty}_{(1/v)}(\rd)$ for 
%some submultiplicative $v$. In particular,
%$f$ has the expansion
%%
\begin{align*}
f &= \sum _{\iota\in \Lambda} \sum _{j\in \Lambda _E} (V_{\psi}f)(j,\iota) \phi(\cdo -j)
e^{i\scal \cdo \iota},
\intertext{and}
\nm f{W^\infty (\omega ,\mascB  )}
&\asymp \sup _{j\in \Lambda _E} \nm {(V_{\psi}f)(j,\cdo )}{\mascB _1},
\quad 
\mascB _1 = \ell ^{\mabfq }_{\mathcal O_0,\tau}(\Lambda )
\end{align*}
for some $\phi \in \Sigma_1(\rd)\setminus 0$, $\psi\in M^{r}_{(v)}(\rd)\setminus 0$
and sufficiently dense lattice $\Lambda \subseteq \rd$, which are fixed (cf.
\cite[Theorem 3.7]{Toft15}). We may
assume that $\Lambda _E'\subseteq \Lambda$.

\par

By using the $E$-periodicity of
$f$, we get by straight-forward computations that
$$
(V_{\psi}f)(j,\iota) =e^{-i\scal j\iota}(V_{\psi}f)(0,\iota)
$$
which gives 
\begin{equation*}
(V_{\psi}f)(x,\xi)= \sum _{\iota \in \Lambda}\sum _{j\in \Lambda_E} (V_{\psi}f)(0,\iota)
(V_{\psi} \phi)(x-j,\xi-\iota)e^{-i\scal j\xi}.
\end{equation*}
In particular, by Proposition \ref{stftGelfand2} we get
\begin{equation*}
|V_{\psi}f(x,\xi)|\lesssim \sum _{\iota \in \Lambda}\sum _{j\in \Lambda _E} |(V_{\psi}f)(0,\iota)|
e^{-R(|x-j| + |\xi -\iota |)},
\end{equation*}
for every $R>0$, giving that
$$
|(V_{\psi}f)(x,\xi )|
\lesssim \sum _{\iota\in \Lambda} |(V_{\psi}f)(0,\iota)| e^{-R|\xi-\iota |}
$$
for every $R>0$. By \eqref{Eq:GSFtransfChar} and
Remark \ref{Rem:FourCoeff} we get 
\begin{multline*}
|c(f,\alpha)| \lesssim \int _E
\left (
\int _{\rd} |(V_{\psi}f)(x,\xi )| |\widehat {\phi}(\alpha -\xi )| \, d\xi 
\right ) \, dx
\\[1ex]
\lesssim \sum _{\iota\in \Lambda} |(V_{\psi}f)(0,\iota )| \int _{\rd} e^{-R|\xi-\iota |}
|\widehat \phi (\alpha - \xi)| \, d\xi
%\\[1ex]
\lesssim \sum _{\iota\in \Lambda} |(V_{\psi}f)(0,\iota )| e^{-R|\alpha -\iota |}
\end{multline*}
for every $R>0$. By choosing $R$ large enough it follows that $v(x)\lesssim
e^{R|x|/2}$. For such $R$ we have 
\begin{multline}\label{Eq:DiscConvEst}
\nm f{\maclE^{E}(\omega,\mascB _0)} 
=
\nm { \{ c(f,\alpha) \omega (\alpha) \} _{\alpha \in \Lambda '_{E}}} {\mascB _0}
\\[1ex]
\lesssim
\nm {|(V_{\psi}f)(0,\cdo )\omega | * (e^{-R|\cdo |} v)} {\mascB _0},
\\[1ex]
\le
\nm {|(V_{\psi}f)(0,\cdo )\omega | * (e^{-R|\cdo |} v)}{\mascB _1}
\\[1ex]
\le
\nm {(V_{\psi}f)(0,\cdo ) \omega}
{\mascB _1} \nm {e^{-R|\cdo |} v}{\ell ^{r}(\Lambda )}
%\\[1ex]
\asymp \nm {(V_{\psi}f)(0,\cdo )\omega} {\mascB _1}
\\[1ex]
=\sup _{j\in \Lambda _E} \nm {(V_{\psi}f)(j,\cdo )\omega} {\mascB _1}
\asymp \nm f {W^{\infty}(\omega,\mascB )}.
\end{multline}
Here $*$ is the discrete convolution with respect to $\Lambda$, and
the second inequality follows from the fact that $\Lambda _E'\subseteq
\Lambda$. This gives the first embedding in \eqref{Embeddings}.
%
%\par
%
%We now extend the definition of
%$\iota \mapsto c_0 (\iota )\equiv (V_{\psi}f)(0,\iota )$ on $\Lambda _E'$
%into a function on $\Lambda$, by letting $c_0(\iota )=0$ when 
%$\iota \in \Lambda \setminus \Lambda _E$.
%Then we still have that \eqref{Eq:DiscConvEst} holds and Young's inequality
%applied to the  right-hand side gives 
%%%
%\begin{multline*}
%\nm f {\maclE ^E (\omega, \mascB _0)} \lesssim \nm {c_0 \omega}
%{\mascB_0} \nm {e^{-R|\cdo |} v}{\ell ^{r}}
%\\[1ex]
%\asymp \nm {c_0\omega} {\mascB _0}
%=\sup _{j\in \Lambda _E} \nm {c(f,j,\cdo)\omega} {\mascB _0}
%\asymp \nm f {W^{\infty}(\omega,\mascB )},
%\end{multline*}
%%%
%and the first embedding in \eqref{Embeddings} follows.

\par

In order to prove the second embedding in \eqref{Embeddings} we observe that 
\begin{equation*}
(V_{\phi} f)(x,\xi) = \sum _ {\alpha\in \Lambda' _E} c(f,\alpha) 
\overline{\widehat{\phi}(\alpha -\xi )}e^{i \scal x{\alpha - \xi}} .
\end{equation*}
Let
\begin{align*}
f_0 (\beta) &= \sup _{x\in \rd} \left (\sup _{\xi\in \beta +E '} |V_{\phi}f(x,\xi)|\right ),
\quad \beta\in \Lambda'_E
\intertext{and}
\kappa (\beta) &= \sup _{\xi\in \beta +E '}|\widehat{\phi}(\xi)| ,
\quad \beta\in \Lambda ' _E.
\end{align*}
Then $\nm f {M^{\infty}(\omega,\mascB )} \asymp \nm {f_0\cdo \omega} {\mascB _0}$, 
and for every $R>0$ we have 
\begin{multline*}
\nm f {M^{\infty}(\omega,\mascB)}
\lesssim 
\Nm {\sum _{\alpha\in \Lambda _E'} |c(f,\alpha)| |\kappa (\alpha -\cdo )\omega |} {\mascB_0}
\\[1ex]
\lesssim
\nm {|c(f,\alpha) \omega | * (e^{-R|\cdo|}v)}{\mascB _0}
\\[1ex]
\leq
\nm {c(f,\cdo )\omega } {\mascB _0} 
\nm {e^{-r|\cdo | }v}{L ^{\min \{1,r \} }}
\asymp
\nm {|c(f,\cdo )\omega |} {\mascB _0} .
\end{multline*}
This gives the result.
\end{proof}

\par

\subsection{Duality properties of $M^\infty (\omega ,\mascB)
\bigcap (\mascE _0^E)'(\rr d)$}

\par

%Suppose $E\subseteq \rr d$ is a non-degenerate parallelepiped,
%spanned by $e_1,\dots ,e_d\in \rr d$, $\omega$ is a weight on $\rr d$
%$$
%\mabfq =(q_1,\dots , q_d)\in (0,\infty ]^d,
%$$
%$\operatorname {S}_d$ be the set of permutations on $\{ 1,\dots ,d\}$,
%and let $f\in \ell ^\infty _{loc}(\Lambda _E)$. As in \cite{Toft10}, let
%$$
%\nm f{\ell ^{\mabfq}_{\sigma ,(\omega )}} \equiv
%\nm {g_{d-1,\omega }}{\ell ^{q_d}(\mathbf Ze_d)},
%$$
%where $g_{d-1,\omega }$ is defined by the formulae
%%%
%\begin{align*}
%g_{0,\omega}(j_1,\dots ,j_d)
%&\equiv |f (j_{\sigma ^{-1}(1)},\dots ,j_{\sigma ^{-1}(d)})
%\omega (j_{\sigma ^{-1}(1)},\dots ,j_{\sigma ^{-1}(d)})|,
%\intertext{and}
%g_{k,\omega}(j_{k+1},\dots ,j_d) &\equiv \nm {g_{k-1,\omega }(\cdo ,
%j_{k+1},\dots ,j_d) }
%{\ell ^{q_k}(\mathbf Ze_k)},
%\quad k=1,\dots ,d-1 .
%\end{align*}
%%%
%We set
%$$
%\max \mabfq =\max (q_1,\dots q_d)
%\quad \text{and}\quad
%\min \mabfq =\min (q_1,\dots q_d).
%$$
%If in addition $\mabfq = [1,\infty ]^d$, then let $\mabfq '=(q_1',\dots ,q_d')$, where
%$q_k'\in [1,\infty ]$ should satisfy $\frac 1{q_k}+\frac 1{q_k'}=1$ for every $k=1,\dots ,d$.
%
%\par
%
%We let $\ell ^{\mabfq}_{\sigma ,(\omega )}(\Lambda _E)$ be the set of all
%$f\in \ell ^\infty _{loc}(\Lambda _E)$ such that $\nm f{\ell ^{\mabfq}_{\sigma ,(\omega )}}<\infty$.
%We also set $\ell ^{\mabfq}_\sigma =\ell ^{\mabfq}_{\sigma ,(\omega )}$ when $\omega =1$.
%The spaces $L^{\mabfq}_{\sigma ,(\omega )}(\rr d)$ and $L^{\mabfq}_\sigma (\rr d)$ are defined
%analogously (cf. \cite{Toft10}).

\par

We begin with the following duality result.

\par

\begin{thm}\label{Thm:PerDual}
Let $E\subseteq \rr d$ be a non-degenerate parallelepiped, $\mathcal O_0$
be as in Definition \ref{Def:MixedLebSpaces}, $\mabfq \in [1,\infty ]^d$,
$\tau \in \operatorname{S_d}$, and let $\omega\in \mascP _E(\rr d)$.
%
%Let $E\subseteq \rr d$ is a non-degenerate parallelepiped, $\sigma \in
%\operatorname {S}_d$, $\omega$ be a weight on $\rr d$, and
%let $\mabfq \in [1,\infty]^d$.
%
Then the following is true:
\begin{enumerate}
\item The form $(\cdo ,\cdo )_E$ from $\maclE _0^E(\rr d)\times \maclE _0^E(\rr d)$
to $\mathbf C$ extends to a continuous map from
$$
\maclE ^E(\omega , \ell ^{\mabfq }_{\mathcal O_0,\tau}(\Lambda _E') )\times
\maclE ^E(1/\omega , \ell ^{\mabfq '}_{\mathcal O_0,\tau}(\Lambda _E') )
$$
to $\mathbf C$. If in addition $\min \mabfq >1$ or $\max \mabfq <\infty$, then the
extension is unique.

\vrum

\item if $\max \mabfq <\infty$, then the dual of
$\maclE ^E(\omega , \ell ^{\mabfq }_{\mathcal O_0,\tau}(\Lambda _E') )$ can be identified by
$\maclE ^E(1/\omega , \ell ^{\mabfq }_{\mathcal O_0,\tau}(\Lambda _E') )$ through the form
$(\cdo ,\cdo )_E$.
\end{enumerate}
\end{thm}

\par

\begin{proof}
By the definitions we may identify $\maclE ^E(\omega , \ell ^{\mabfq }
_{\mathcal O_0,\tau}(\Lambda _E'))$ and the form $(\cdo ,\cdo )_E$
with $\ell ^{\mabfq }_{\mathcal O_0,\tau}(\Lambda _E')$ and the 
form $(\cdo ,\cdo )_{\ell ^2(\Lambda _E')}$. The result now follows from
the fact that similar properties hold true for mixed normed Lebesgue spaces.
\end{proof}

\par

\begin{cor}
Let $E\subseteq \rr d$ be a non-degenerate parallelepiped, $\mathcal O_0$
be as in Definition \ref{Def:MixedLebSpaces}, $\mabfq \in [1,\infty )^d$,
$\tau \in \operatorname{S_d}$, and let $\omega\in \mascP _E(\rr d)$.
%
%Let $E\subseteq \rr d$ is a non-degenerate parallelepiped, $\sigma \in
%\operatorname {S}_d$, $\omega \in \mascP _E(\rr d)$, and
%let $\mabfq \in [1,\infty )^d$.
%
Then the form $(\cdo ,\cdo )_E$
from $\maclE _0^E(\rr d)\times \maclE _0^E(\rr d)$
to $\mathbf C$ extends uniquely to a continuous map from
$$
M^\infty (\omega , L^{\mabfq }_{\mathcal O_0,\tau}(\rr d) )
\bigcap (\maclE _0^E)'(\rr d)\times
M^\infty (1/\omega , L^{\mabfq '}_{\mathcal O_0,\tau}(\rr d) )
\bigcap (\maclE _0^E)'(\rr d)
$$
to $\mathbf C$, and the dual of
$M^\infty (\omega , L^{\mabfq }_{\mathcal O_0,\tau}(\rr d) )\bigcap (\maclE _0^E)'(\rr d)$
can be identified by $M^\infty (1/\omega , L^{\mabfq '}_{\mathcal O_0,\tau}(\rr d) )
\bigcap (\maclE _0^E)'(\rr d)$ through this form.

\par

In particular, if $q\in [1,\infty )$ and $\omega _0(x,\xi )=\omega (\xi )$, then the dual
of $M^{\infty ,q}_{(\omega _0)} (\rr d)\bigcap (\maclE _0^E)'(\rr d)$
can be identified by $M^{\infty ,q'}_{(1/\omega _0)} (\rr d)\bigcap (\maclE _0^E)'(\rr d)$.
\end{cor}

\begin{proof}
The result follows by combining Theorems \ref{Thm:PerDual} with \ref{Thm:PerMod}.
\end{proof}

\par

\end{document}